\documentclass{amsart}

\usepackage [mathscr]{eucal}
\usepackage{amsmath,amsthm,amssymb,amscd,color}
\usepackage{amsrefs}
\usepackage{epsf}
\usepackage{amsfonts}
\usepackage{dsfont}
\usepackage{latexsym}
\usepackage{mathrsfs}
\usepackage{layout}
\usepackage{bm}
\usepackage{epsfig}
\usepackage{subfigure}
\usepackage{here}
\usepackage{caption}
\usepackage{graphics}
\usepackage{graphicx}
\usepackage{comment}
\usepackage{verbatim}

\usepackage{mathtools}
\DeclarePairedDelimiter{\ceil}{\lceil}{\rceil}

\usepackage{geometry}
\newtheorem{theorem}{Theorem}[section]
\newtheorem{lemma}[theorem]{Lemma}
\newtheorem{proposition}[theorem]{Proposition}
\newtheorem{corollary}[theorem]{Corollary}

\newtheorem{remark}[theorem]{Remark}
\newtheorem{definition}[theorem]{Definition}

\numberwithin{equation}{section}
\newtheorem{mainthm}{Theorem}

\newcommand{\R}{{ \mathbb{R}  }}

\newcommand{\abs}[1]{\lvert#1\rvert}
\newcommand{\bke}[1]{\left( #1 \right)}

\newcommand{\norm}[1]{\left\Vert #1 \right\Vert}

\def\XXint#1#2#3{{\setbox0=\hbox{$#1{#2#3}{\int}$ }
\vcenter{\hbox{$#2#3$ }}\kern-.6\wd0}}

\def\BigRoman{\uppercase\expandafter{\romannumeral\number\count255 }}
\def\Romannumeral{\afterassignment\BigRoman\count255=}

\DeclareMathOperator *{\essosc}{ess\ osc}
\DeclareMathOperator *{\osc}{osc}
\DeclareMathOperator *{\esssup}{ess\ sup}
\DeclareMathOperator *{\essinf}{ess\ inf}
\DeclareMathOperator *{\di}{div} 
\DeclareMathOperator *{\meas}{meas}
\DeclareMathOperator *{\dist}{dist}
\DeclareMathOperator *{\data}{data}
\DeclareMathOperator *{\diam}{diam}
\DeclareMathOperator *{\loc}{loc}
\DeclareMathOperator *{\Lip}{Lip}
\DeclareMathOperator *{\Tr}{Tr}
\DeclareMathOperator *{\BMO}{BMO}
\DeclareMathOperator *{\Proj}{Proj}

\newenvironment{thm1.5}{{\par\noindent\bf
           Proof of Theorem \ref{weak-1}. }}
           {\hfill\fbox{}\par\vspace{.2cm}}

\newenvironment{thm1.6}{{\par\noindent\bf
           Proof of Theorem \ref{weak-2}. }}
           {\hfill\fbox{}\par\vspace{.2cm}}

\newenvironment{thm1.7}{{\par\noindent\bf
           Proof of Theorem \ref{weak-4}. }}
           {\hfill\fbox{}\par\vspace{.2cm}}

\newenvironment{thm1.8}{{\par\noindent\bf
           Proof of Theorem \ref{bdd weak-1}. }}
           {\hfill\fbox{}\par\vspace{.2cm}}

\newenvironment{thm1.9}{{\par\noindent\bf
           Proof of Theorem \ref{bdd weak-2}. }}
           {\hfill\fbox{}\par\vspace{.2cm}}

\newenvironment{thm1.10}{{\par\noindent\bf
           Proof of Theorem \ref{bdd weak-4}. }}
           {\hfill\fbox{}\par\vspace{.2cm}}

\begin{document}
\title[Degenerate diffusion-drift equation]{Continuity results for \\
degenerate diffusion equations with $L^{p}_t L^{q}_{x}$ drifts}

\author{Sukjung Hwang}
\address{
(S. Hwang)  Center for Mathamatical Analysis and Computation, Yonsei University}
\email{sukjung{\_}hwang@yonsei.ac.kr, sukjungh@gmail.com}
\thanks{S. Hwang is partially supported by Basic Science Research Program
        through the National Research Foundation of Korea(NRF) funded
        by the Ministry of Education, Science and Technology(2017R1D1A1B03035152).}

\author{Yuming Paul Zhang}
\address{
(Y. Zhang) Department of Mathematics \\ 
University of California   \\ 
Los Angeles\\
USA}
\email{yzhangpaul@math.ucla.edu}



\begin{abstract}
In this paper, we study local uniform continuity of nonnegative weak solutions to degenerate diffusion-drift equations in the form
\[
u_{t} = \Delta u^{m} + \nabla\cdot \left( B (x,t) \, u\right), \quad \text{for } m \geq 1
\]
assuming a vector field $B \in L^{p}_t L^{q}_{x}$.
{Regarding local H\"{o}lder continuity, we provide a sharp condition on $p$ and $q$, which is referred to as the subcritical region. In the critical region, the divergence-free condition is essential to providing uniform continuity which depends on the modulus continuity of $B$.} 
\end{abstract}

\maketitle

\section{Introduction}


In this paper, we concern the following form of equation:
\begin{equation}\label{PME}
u_{t} = \Delta u^{m} + \nabla\cdot \left( B (x,t) \, u\right)
\end{equation}
{for $m \geq 1$ and here $B$ is a vector field.
Our goal is to study the local continuity property of $u$, a bounded non-negative weak solution of \eqref{PME} in the domain $Q_1:=K_1\times (-1,0]$ where $K_1:=\{x\in \mathbb{R}^d,\,|x|< 1\}$, under the condition
\begin{equation}\label{B}
B \in {L^{2\hat{q}_1}_t L^{ 2\hat{q}_2}_{x}} \left( Q_1\right).
\end{equation}
for some positive constants $\hat{q}_1 > 1$ and $\hat{q}_2 > 1$.} This result handles both linear $(m=1)$ and degenerate $(m>1)$ equations under general $B$ that can depend on both spatial and time variables. Moreover, the continuity result for a linear parabolic equation can be extended to cover a uniform parabolic equation, Remark~\ref{rmk linear}.

\medskip

By rewriting $\Delta u^m$ in \eqref{PME} into the divergence form $\nabla \cdot \left(m \, u^{m-1} \nabla u\right)$, one can easily see the parabolicity $m u^{m-1}$ shows different tendencies depending on the range of $m$. If $m=1$, then \eqref{PME} has linear diffusion. When $m>1$, \eqref{PME} carries degenerate diffusion because the parabolicity tends to zero as $u$ approaches to zero. For $0<m<1$, the parabolicity tends to blow up when $u$ tends to zero and such a equation is called fast diffusion equation and has singular diffusion.

{When $B=0$, \eqref{PME} is the well-known \emph{Porous Medium Equation} { ($PME$)} which describes a number of physical applications:} fluid flow \cites{Lei, Muskat, Bouss}, heat radiations in plasma \cite{Zeldovich}, mathematical biology \cites{CalCar, Kowal}, and other fields.
For a systematic presentation of the mathematical theories, we refer to \cites{Va, DBGiVe06}. The Harnack inequalities and H\"{o}lder continuity are explained in \cite{DBGiVe06} in details following the spirit of DeGiorgi type iterations using the intrinsic geometric properties. It is already known that a bounded non-negative weak solution of \eqref{PME} with $B=0$ is locally H\"{o}lder continuous.

In case $B \neq 0$, \eqref{PME} represents a certain form of reaction-diffusion equation which corresponds to several physical phenomena: for example, local chemical reactions in which the substances are transformed into each other and diffusion which causes the substances to spread out in space \cite{KellerSegel}. Especially, the divergence-free condition is relevant for applications to incompressible fluids \cite{Osada}.

\medskip

By a scaling argument, which will be presented later, it can be shown that the following identity is crucial for all $m\geq 1$:
\begin{equation}\label{hatq12}
\frac{2}{\hat{q}_1} + \frac{d}{\hat{q}_2} = 2 - d\kappa
\end{equation}
for some $\kappa \in [0, 2/d)$.
We allow $\hat{q}_1=\infty$ and in such case, $\hat{q}_2$ becomes $\frac{d}{2-d\kappa}$.
We will refer to the case $\kappa>0$ as \textit{subcritical} and $\kappa=0$ as the \textit{critical} region.

\medskip

The properties of weak solutions of uniform parabolic diffusion-drift equations were investigated in a number of papers.
For uniformly parabolic equations in the form of
\begin{equation*}
    u_t - \di \left( A(x,t) \nabla u \right) + B(x,t) \nabla u = 0,
\end{equation*}
H\"{o}lder estimates were obtained by Ladyzhenskaya and Ural'tseva \cite{LU} for $b \in L^{q+2}_{x}$, $q>d$. In \cite[Chapter~VI]{Lieberman}, Lieberman showed Harnack's inequality under $b \in L^{p}_{t}L^{q}_{x}$ with $2/p + d/q <1$, and obviously Harnack's inequality implies H\"{o}lder continuity.
For the borderline case with $2/p + d/q =1$, under the additional condition $\di B\leq 0$, Nazarov and Ural'tzeva \cite{NU11} proved H\"{o}lder continuity of the solution. Safanov \cite{Safanov} solved the problem in the critical region of corresponding linear stationary equations. However, as far as we are aware, it is not discussed for parabolic equations with general drifts.

Aiming the same regularity, often divergence-free drift requires weaker conditions.
With divergence-free $B$,
Osada \cite{Osada} proved H\"{o}lder continuity of solutions assuming $B \in L^{\infty}_{t} (L^{\infty}_{x})^{-1}$. In \cite{Zhang}, the vector field is allowed to be much more singular than boundedness in $L^d(\mathbb{R}^d)$.
Friedlander and Vicol \cite{FV} showed H\"{o}lder continuity assuming divergence-free $B \in L_{t}^{\infty}\text{BMO}_{x}^{-1}$. Let us mention that $L^{\infty}_{t}(L^{\infty}_{x})^{-1}$, $L^{d}_{x}$, $\text{BMO}_{x}^{-1}$ are all sharing the same scaling property.
{In general, people belive that the divergence-free assumption only provides a borderline improvement \cites{SVZ}.} Seregin, Silvestre, Sverak, and Zlatos \cite{SSSZ}:  in two dimensions, even $L^1$ bound for time-independent divergence-free drift turns out to be sufficient to yield continuous solutions with a weak modulus of continuity. The authors also constructed counterexamples saying that solutions $u\in L^{\infty} \cap H^1$ with divergence-free $B\in L^{1}_{x}$ may be discontinuous in three and higher dimensions.


While for degenerate equations, there are fewer results.
In DiBenedetto's work (see \cite{DB83} and \cite[Section IV.15]{DB93}), H\"{o}lder regularity estimates hold for the following equation
\begin{equation}\label{dib equation}
    u_t - \Delta u^m + b(x,t,u,\nabla u) = 0\quad \text{ where }|b(x,t,u,\nabla u)| \leq C \left|\nabla u^m\right|^{2} + \varphi(x,t)
\end{equation}  for a function $\varphi$ bounded in some Lebesgue spaces.

On the other hand, in \cite{CHKK}, the equation \eqref{PME} for $m\geq 1$ and its regularity under the condition on $B \in L^{2\hat{q}_1}_t L^{ 2\hat{q}_2}_{x} $ arises naturally to understand the motion of swimming bacteria modeled by Keller-Segel system with porous medium type diffusion coupled to fluid equations, such as \emph{ Bacillus subtilis} \cite{Tuval}. In \cite{CHKK}, authors proved the H\"{o}lder continuity of \eqref{PME} under the condition on $B(x,t)$ and $\nabla B (x,t)$ such as $B \in L^{2\hat{q}_1}_t L^{ 2\hat{q}_2}_{x} $ and $\nabla B \in L^{\hat{q}_1}_t L^{\hat{q}_2}_{x}$ where $\hat{q}_1$ and $\hat{q}_2$ satisfies \eqref{hatq12}. 

This paper is motivated by one of recent works by Kim and Zhang \cite{KZ} studying the well-posedness and continuity of a weak solution of \eqref{PME} with $m>1$ and $B\in L_t^\infty L_x^p$. In \cite{KZ}, the authors construct counterexamples that fail the boundedness and the uniform modulus of continuity if $p\leq d$. Moreover, the boundedness of a weak solution is shown for the case $p>d$, while the H\"{o}lder continuity is only proved if $p > d + \frac{4}{d+2}$ which leaves a question mark on the H\"{o}lder continuity in the range $ d < p \leq d +\frac{4}{d+2}$. In \eqref{hatq12}, one can easily observe that $2\hat{q}_2 \in (d, \infty]$ when $\hat{q}_1 = \infty$. In this paper, we are able to extend H\"{o}lder estimate for all $p>d$ completing the sharpness by applying the same DeGiorgi-Nash-Moser iteration method with more efficient use of the Sobolev embedding inspired by \cite{DB93}.

\medskip

Now let us revisit the condition on $B$ in \eqref{B} through a scaling argument. {Suppose that $\omega > 0$ is approximately the essential oscillation of $u$ in a cylinder $Q_{\omega,r}'=K_r\times (-w^{1-m}r^2,0]$.}
Let $u_{\omega, r}(x,t): = \omega^{-1} u(rx, \omega^{1-m}r^2 t)$. This scaling preserves the zero-drift equation ($PME$) and the rescaling for the time variable is often referred to as the \emph{ intrinsic scaling} which is essential in understanding local behavior of nonlinear parabolic equations \cites{DB86, DBGiVe06}.  This $u_{\omega, r}$ solves
\[
\partial_{t} u_{\omega, r} = \Delta u_{\omega, r}^{m} + \nabla \cdot \left( B_{\omega, r} u_{\omega, r}\right).
\]
for
\[
B_{\omega, r}(x,t): = \omega^{1-m} r B(rx, \omega^{1-m}r^2 t).
\]
Then we make an observation that
\begin{align*}
\| B_{\omega, r}\|_{L_{t}^{2\hat{q}_1}L_{x}^{2\hat{q}_2}{(Q_{\omega,r}')}}
&= \left(\int_{-\omega^{1-m}r^2}^{0} \left[ \int_{K_r} B_{\omega, r}^{2\hat{q}_2} (y, s)  \,dy \right]^{\frac{2\hat{q}_1}{2\hat{q}_2}} \,ds \right)^{\frac{1}{2\hat{q}_2}} \\
&= \omega^{1-m}r\left(\int_{-1}^{0} \left[ \int_{K_1}  B^{2\hat{q}_2} (x,t)  r^{-d}\,dx\right]^{\frac{2\hat{q}_1}{2\hat{q}_2}} \omega^{m-1}r^{-2} \,dt \right)^{\frac{1}{2\hat{q}_2}}\\
& = \omega^{(1-m)\left(1-\frac{1}{2\hat{q}_1}\right)} r^{1- \frac{1}{2}\left( \frac{2}{\hat{q}_1} + \frac{d}{\hat{q}_2}\right)} \| B\|_{L_{t}^{2\hat{q}_1}L_{x}^{2\hat{q}_2}(Q_1)}.
\end{align*}
{Without loss of generality, we regard $\omega$ as a positive constant which is comparably large to $r$ (by which we mean $\omega\geq r^\epsilon$ for any small $\epsilon>0$), because otherwise we might already have good regularity of $u$ in the scale $r$.} {Now as $r\to 0$ if $\| B_{\omega, r}\|_{{L_{t}^{2\hat{q}_1}L_{x}^{2\hat{q}_2}}}\to 0$, we expect nice behaviour of the solution since the equation becomes ($PME$)  when $r=0$}. This leads to the condition
\begin{equation}\label{hatq12:si}
1- \frac{1}{2}\left( \frac{2}{\hat{q}_1} + \frac{d}{\hat{q}_2}\right) \geq 0,
\end{equation}
that matches to \eqref{hatq12}. The strict inequality is the {subcritical region} and the equality is the {critical case}.

\medskip

It turns out to be more convenient to consider {$v := u^{m}$} instead of $u$, and the equation \eqref{PME} becomes
\begin{equation}
\label{main v}
\partial_t(v^{1/m})=\Delta v+\nabla \cdot(B\,v^{1/m}).
\end{equation}
We study the continuity of $v$ a non-negative bounded weak solution of \eqref{main v} in $Q_1$ and then recover the continuity of $u$.

\medskip


Below we state two Theorems A-B that summarize our main results. The first result is given for the subcritical region.

\begin{mainthm}[Theorem~\ref{T:Holder:v}] \label{Theorem1} 
Suppose \eqref{hatq12} holds with $\kappa \in (0, d/2)$.
 Let $u$ be a non-negative bounded weak solution of \eqref{PME} with $m \geq 1$ in $Q_1$. Then $u$ is uniformly H\"{o}lder continuous in $Q_\frac{1}{2}=K_{\frac{1}{2}}\times (-\frac{1}{4},0]$.
\end{mainthm}

{In particular, when $m=1$, we recover the result of Lieberman. However, our result also shows that as $m\to 1$ from the above, uniform H\"{o}lder regularity holds.}

We modify DiBenedetto's method (see \cite{DB83}) of treating $\varphi \in L_{t}^{\hat{r}}L_{x}^{\hat{q}}$ in \eqref{dib equation} to handle our drift term given $\nabla (B u)$ for $B \in L_{t}^{2\hat{q}_1}L_{x}^{2\hat{q}_2}$.
{We prove two alternatives, which mainly says that if restricting to a smaller region, either the maximum of the solution decreases or the minimum increases depending on its local average. From this, we derive the reduction of the oscillation of the solution which, after iterations, leads to a uniform continuity property of the solution.
The hard part we need to overcome while practicing this argument is to control the effects coming from the unbounded drift with the help of its norm in the subcritical spaces.}

\medskip

Recall Theorem~1.3 from \cite{KZ}. We have
\begin{theorem}\label{thm Ac}
 For $d=2,3$,
there is a sequence of divergence-free vector fields $\{B_n\}_n$ that are uniformly bounded in $L^d(\R^d)$, and a sequence of uniformly smooth initial data $\{u_{n,0}\}_n$, such that the corresponding solutions $\{u_n\}$  of \eqref{PME} are uniformly bounded in height but not bounded in any H\"{o}lder norm in a unit parabolic neighborhood.
\end{theorem}
Thus as for uniform H\"{o}lder regularity, the condition in Theorem \ref{Theorem1} on the drifts given by \eqref{hatq12} for $\kappa>0$ is sharp. Seeing from the theorem, the loss of the uniform H\"{o}lder regularity happens in finite time even for divergence-free drifts.

\medskip

Second, we consider divergence-free drift $B$ which lies in the critical regime and we study the uniform continuity of weak solutions of \eqref{PME}. The continuity depends on the modulus of continuity of the local $L_t^{2\hat{q}_1}L_x^{2\hat{q}_2}$ norms of $B$. Define
\begin{equation}
    \label{def varrho}
    \varrho_B(r):=\sup_{x_0,t_0}\left\|B\right\|_{L_t^{2\hat{q}_1}L_x^{2\hat{q}_2}((x_0,t_0)+Q_{r})\cap Q_1)}\quad \text{ with }Q_r:=K_{r}\times (-r^2,0].
\end{equation}
By definition, as $r\to 0$, we have $\varrho_B(r)\to 0$ and therefore $\varrho_B$ is a modulus of continuity.


Now we present our result:

\begin{mainthm}\label{Theorem2}
Results in the critical region.
  \begin{enumerate}
    \item [(a)] [Theorem~\ref{main cri}]
    Suppose $B$ is divergence-free and belongs to $L_t^{2\hat{q}_1}L_x^{2\hat{q}_2}(Q_1)$ for \eqref{hatq12} with $\kappa = 0$. Let $u$ be a non-negative bounded weak solution of \eqref{PME} with $m \geq 1$ in $Q_1$. Then $u$ is uniformly continuous in $Q_\frac{1}{2}$ depending on $m,d, \hat{q}_1, \hat{q}_2$, $\|u\|_{L^\infty(Q_1)}$ and $\varrho_B(r)$ as in \eqref{def varrho}. 
\item [(b)] [Theorem~\ref{thm:lackofcont}]
     For any fixed $m\geq 1$, there exists a sequence of vector fields $\{B_n\}_n$, which are uniformly bounded in $L^d(K_1)$ and
 \begin{equation}
     \label{mainthm B}
     \varrho_{B_n}(r)\leq \omega(r) \quad \text{for some modulus of continuity }\omega,
 \end{equation}
 along with a sequence of uniformly bounded functions $\{u_{n}\}$ in $K_1$ which are stationary solutions to \eqref{PME} with $B=B_n$ such that, they do not share any common mode of continuity.
  \end{enumerate}
\end{mainthm}

For the regularity result, again we follow the idea of DiBenedetto's method.
Unfortunately, when we assume $\kappa =0$, it fails the energy estimates and the proof of the alternatives developed for the subcritical regime. Indeed, the divergence-free condition on $B$ is essential to derive a proper energy estimate and new techniques based on DiBenedetto's method are used for the regularity result.

By the counterexample, we know that in the critical region, non-divergence-free condition on $B$ prevents any mode of continuity for all $m\geq  1$ even for restricted modulus of continuity of the local norm of $B$. The divergence-free condition in the theorem is sharp for linear diffusion cases. Part (b) of the theorem can be viewed as an extension of Theorem~5.2 from \cite{KZ} where \eqref{mainthm B} is not imposed.

{Finally let us mention that, for the divergence-free drift fractional diffusion equation, \cite{SVZ} discussed the loss of continuity in the supercritical region. As for our equation, the supercritical region is defined with $\kappa<0$ in \eqref{hatq12}. It is not known in such cases whether or not there exist discontinuous solutions to the equation with general drifts particularly for $m>1$, though we believe the answer is positive. And of course, for divergence-free drifts, the continuity problem in the supercritical region is open.}

\medskip

This paper is composed of the following sections. Section~\ref{Preliminaries} is for notations and known theories and inequalities. Next Section~\ref{S:Energy} is to deliver all energy estimates for both subcritical and critical regimes, which are the starting points to pursue the continuity results in later sections.  Section~\ref{S:subcritical} is devoted to deliver H\"{o}lder continuity in the subcritical regime. The last Section~\ref{S:critical} is about the critical regime that the divergence-free condition on drifts is sharp in terms of the uniform continuity.

\section{Preliminaries}\label{Preliminaries}

\subsection{Notations and useful inequalities} In this subsection,
we introduce the notations throughout this paper and recall some
useful inequalities for our purpose.

For future use, we denote
\begin{equation}\label{q12}
q_1 = \frac{2\hat{q}_1 (1+\kappa)}{\hat{q}_1 - 1}, \quad q_2 = \frac{2\hat{q}_2 (1+\kappa)}{\hat{q}_2 - 1},
\end{equation}
which is a pair of a dual of $(2\hat{q}_1,2\hat{q}_2)$.

Let $\Omega$ be an open domain
in $\mathbb{R}^{d}$, $d\ge 1$ and $I$ a finite interval. We denote
\[
L^{p}(\Omega)=\{ f : \Omega \rightarrow \mathbb{R} \mid f \text{ is
Lebesgue measurable, } \norm{f}_{L^{p}(\Omega)}<\infty\},
\]
 where
\[
 \norm{f}_{L^{p}(\Omega)}= \bke{\
 \int_{\Omega}|f|^{p}dx}^{\frac{1}{p}}, \qquad (1 \leq p < \infty).
\]
We will write $\norm{f}_{p,\Omega}:=\norm{f}_{L^{p}(\Omega)}$, unless there
is any confusion to be expected. For $1\leq p \leq \infty$,
$W^{k,p}(\Omega)$ denotes the usual Sobolev space, i.e.,
\[
W^{k,p}(\Omega)=\{u\in L^{p}(\Omega):D^{\alpha}u\in L^{p}(\Omega), 0
\leq |\alpha|\leq k\}.
\]
We also write the mixed norm of $f$ in spatial and temporal
variables as
\[
\norm{f}_{L^q_tL^p_x(\Omega\times I)}=\norm{f}_{L^{q}_{t}(I;
L^p_x(\Omega))}=\norm{\norm{f}_{L^p_x(\Omega)}}_{L^q_t(I)}.
\]
{For $\Omega_{T} := \Omega \times [0, T]$,
we will write $\norm{f}_{p,\Omega_T}:=\norm{f}_{L_t^{p}L_x^p(\Omega_T)}$ for simplicity.}

Take $p\geq 1$ and consider
the Banach spaces
\[
V^{p}(\Omega_{T}) := L^{\infty} (0, T; L^{p}(\Omega)) \cap
L^{p}(0,T; W^{1,p}(\Omega))
\]
and
\[
V^{p}_{0}(\Omega_{T}) := L^{\infty} (0, T; L^{p}(\Omega)) \cap
L^{p}(0,T; W^{1,p}_{0}(\Omega)),
\]
both equipped with the norm $v \in V^{p}(\Omega_{T})$,
\[
\| v\|_{V^{p}(\Omega_{T})} : = \esssup_{0<t<T} \|v(\cdot, t)\|_{p,
\Omega} + \| \nabla v\|_{p, \Omega_{T}}.
\]

Note
that if $p$ is finite, both spaces are embedded in $L^{q}(\Omega_{T})$ for some $q >
p$. We denote by $C=C(\alpha,\beta,...)$ a constant depending on the
prescribed quantities $\alpha,\beta,...$, which may change from line
to line.

Now we introduce basic embedding inequalities.

\begin{theorem}\label{T:GN} (Gagliardo-Nirenberg multiplicative embedding inequality)
Let $v \in W_{0}^{1,p}(\Omega)$, $p\geq 1$. For every fixed number
$s \geq 1$ there exists a constant $C$ depending only upon $d$, $p$
and $s$ such that
\[
\| v\|_{q, \Omega} \leq C \|\nabla v\|^{\alpha}_{p, \Omega}
\|v\|^{1-\alpha}_{s,\Omega},
\]
where $\alpha \in [0,1]$, $p, q \geq 1$, are linked by
\[
\alpha = \left( \frac{1}{s} - \frac{1}{q}\right) \left(\frac{1}{d} -
\frac{1}{p} + \frac{1}{s}\right)^{-1},
\]
and their admissible range is
\[\begin{cases}
q \in [s, \infty], \ \alpha \in [0, \frac{p}{p + s(p-1)}], & \text{ if } d = 1, \\
q\in [s, \frac{dp}{d-p}], \ \alpha \in [0,1], & \text{ if } 1\leq p < d, \ s \leq \frac{dp}{d-p}, \\
q\in [\frac{dp}{d-p}, s], \ \alpha \in [0,1], & \text{ if } 1\leq p < d, \ s \geq \frac{dp}{d-p}, \\
q \in [s, \infty), \ \alpha\in [0, \frac{dp}{dp + s(p-d)}), & \text{
if } 1<d\leq p .
\end{cases}\]
\end{theorem}

\begin{theorem}\label{T:Sobolev} (Sobolev embedding theorem) Let $p>1$. There exists a constant $C$ depending only upon $d, p$ such that for every  ,
\[
\iint_{\Omega_{T}} \left| v(x,t) \right|^{q} \,dxdt \leq C^{q}\left(\esssup_{0<t<T} \int_{\Omega} \left|
v(x,t)\right|^{p} \,dx \right)^{p/d} \left( \iint_{\Omega_{T}}
\left| \nabla v(x,t)\right|^{p}\,dxdt\right)
\]
where $q = \frac{p(d+p)}{d}$.
\end{theorem}
By H\"{o}lder's inequality {and} Sobolev embedding theorem, we have
\begin{corollary}For $ v\in L^{\infty}(0,T; L^{p}(\Omega)) \cap L^{p}(0, T; W^{1,p}_{0}(\Omega))$, there exists $C$ depending only upon $d, p$ such that
\begin{align*}
\|v\|_{q, \Omega_{T}}
&\leq C \|v\|_{ V^{p} (\Omega_{T})}
\intertext{ and }
\| v \|^{p}_{p, \Omega_{T}} &\leq C \left| \{ |v| > 0
\}\right|^{\frac{p}{d+p}} \|v\|^{p}_{V^{p}(\Omega_{T})}.
\end{align*}
\end{corollary}

\begin{proposition}\label{P:admissible}
There exists a constant $C$ depending only upon $d$ and $p$ such
that for every $v\in V^{p}_{0}(\Omega_{T})$,
\[
\| v \|_{L^r_tL^q_x( \Omega_{T})} \leq C \|v\|_{V^{p}(\Omega_{T})}
\]
where the numbers $q, r \geq 1$ are linked by
\[
\frac{1}{r} + \frac{d}{pq} = \frac{d}{p^2},
\]
and their admissible range is
\[\begin{cases}
q \in (p,\infty), \ r\in (p^2, \infty) & \text{ for } d = 1, \\
q \in (p, \frac{dp}{d-p}), \ r\in (p, \infty) & \text{ for } 1 < p < d, \\
q \in (p, \infty), \ r\in (\frac{p^2}{d}, \infty) & \text{ for } 1
< d \leq p .
\end{cases}\]
\end{proposition}

\begin{proof}
Let $v\in V_{0}^{p}(\Omega_{T})$ and let $r \geq 1$ to be chosen.
From Theorem~\ref{T:GN} with $s=p$ follows that
\[
\left( \int_{0}^{T} \|v(\cdot, \tau)\|^{r}_{q, \Omega} \,d\tau\right)^{1/r}
\leq C \left( \int_{0}^{T} \|\nabla v(\cdot, \tau)\|^{\alpha r}_{p}
\,d\tau\right)^{1/r} \esssup_{ 0\leq r \leq T} \| v (\cdot,
\tau)\|^{1-\alpha}_{p, \Omega_{T}}.
\]
Choose $\alpha$ such that $\alpha r = p$.
\end{proof}

We state a lemma concerning the geometric convergence
of sequences of numbers (Refer Chapter~I in \cite{DB93}).


\begin{lemma}\label{L:iter2}
Let $\{Y_n\}$ and $\{Z_n\}$, $n=0,1,2, \ldots,$ be sequences of
nonnegative numbers, satisfying the recursive inequalities
\[\begin{cases}
Y_{n+1} \leq C b^{n} \left( Y_{n}^{1+\alpha} + Z_{n}^{1+\kappa} Y_{n}^{\alpha}\right) & \\
Z_{n+1} \leq C b^{n}\left( Y_{n} + Z_{n}^{1+\kappa} \right) &
\end{cases}\]
where $C,b>1$ and $\kappa, \alpha >0$ are given numbers. If
\[
Y_{0} + Z_{0}^{1+\kappa} \leq
(2C)^{-\frac{1+\kappa}{\sigma}}b^{-\frac{1+\kappa}{\sigma^2}}, \quad
\text{where } \sigma = \min\{\kappa, \alpha\},
\]
then $\{ Y_n \}$ and $\{ Z_n \}$ tend to zero as $n \to \infty$.
\end{lemma}

The following lemma is introduced in \cite{DBGiVe06}; it states that
if the set where $v$ is bounded away from zero occupies a sizable
portion of $K_{\rho}$, then the set where $v$ is positive cluster
about at least one point of $K_{\rho}$. Here we name the inequality
as the isoperimetric inequality.

\begin{lemma}\label{Iso} (Isoperimetric inequality)
Let $v \in W^{1,1}(K_{\rho}) \cap C(K_{\rho})$ for some
$\rho > 0$ and some $x_0 \in \mathbb{R}^{d}$ and let $k$ and $l$ be
any pair of real numbers wuch that $k < l$. Then there exists a
constant $\gamma$ depending upon $N, p$ and independent of $k, l, v,
x_0, \rho$, such that
\[
(l-k) |K_{\rho} \cap \{v > l\}| \leq \gamma
\frac{\rho^{d+1}}{|K_{\rho} \cap \{v \leq k\}|}
\int_{K_{\rho}\cap \{k< v < l\}} |Dv|\,dx .
\]
\end{lemma}

{



Here we introduce the notion of a nonnegative local weak solution of \eqref{PME} (while a global solution is defined in \cite[Definition~2.1]{KZ}):
\begin{definition}\label{D:weak_sol}
Recall $Q_1:=K_1\times (-1,0]$ and $K_1=\{|x|<1\}$. We say that a nonnegative function $u(x,t): \overline{Q}_1 \to [0, \infty)$ is a solution to \eqref{PME} in $Q_1$ if
\begin{gather*}
u \in C\left((-1, 0], L^{1}(K_1)\right) \cap L^{\infty}\left( \overline{Q}_1 \right), \\
 u^m \in L^2\left((-1, 0], W^{1,2} (K_1)\right).
\end{gather*}
And for all nonnegative test functions $\phi \in C^1_{c} \left( Q_1 \right)$, we have for any $T\in (-1,0]$
\[
\int_{-1}^{T}\int_{K_1} u \phi_t \,dx\,dt  =\int_{K_1} u (x,T) \phi(x,T) \,dx+ \int_{-1}^{T}\int_{K_1} \left[ \nabla u^m + u B \right] \nabla \phi \,dx\,dt.
\]
\end{definition}

In this paper, we only concern the local regularity of a nonnegative weak solution of \eqref{PME}.

}

\section{Energy estimate}\label{S:Energy}


{

In this section, we provide local and logarithmic type energy estimates for $v$ a non-negative bounded weak solution of \eqref{main v} that are key to prove the continuity of $v$ under the case the drift term $B$ in \eqref{hatq12} is either subcritical $\kappa \in (0, 2/d)$ or critical $\kappa = 0$ regions. When we discuss $v$ from \eqref{main v}, the constant
\begin{equation}
\beta=\frac{m-1}{m}
\end{equation}
appears in many places such as the intrinsic scaling. In the case $m=1$ of uniform parabolic equations, we have $\beta=0$.

Denote $K_\rho$ as a ball centered at the origin with radius $\rho$.
{Denote the parabolic cylinder
\begin{equation}\label{Qrho}
Q_\rho:=K_\rho\times (t_0,t_1].
\end{equation}
Let us denote the parabolic boundary of $Q_\rho$  as
\[
\partial_{p} Q_{\rho} := \partial K_{\rho} \times [t_0, t_1] \cup K_{\rho}\times \{t_0\}
\]
as the union of the lateral and initial boundaries.
}

Let
\begin{equation}\label{mu+-}
\mu_+:=\sup_{Q_\rho}v, \quad \mu_-:=\inf_{Q_\rho}v.
\end{equation}

Note that both $v-\mu_{-}$ and $\mu_{+} -v$ are non-negative, therefore, we set up the upper and lower level set $ (v - \mu_{\pm} \pm k)_{\pm} $ for a given positive constant $k\in (\mu_-,\mu_+)$.
For simplicity, let
\begin{equation}\label{v_pm}
v_{\pm} := (v - \mu_{\pm} \pm k)_{\pm}.
\end{equation}

For given positive constant $\rho$, we denote the set
\begin{equation}\label{Akrho_t}
A_{k,\rho}^{\pm}(\tau) = \{ x\in K_{\rho}: \, \left(v(x,\tau) -
\mu_{\pm} \pm k \right)_{\pm} > 0 \},
\end{equation}
that indicates a level set (either $ v < \mu_{-} + k$ or $v> \mu_{+}
- k$) at a fixed time $\tau$.
 Moreover, let
\begin{equation}\label{Akrho}
A_{k,\rho}^{\pm} = Q_{\rho} \cap \{\left(v(x,\tau) - \mu_{\pm} \pm k \right)_{\pm} > 0 \}.
\end{equation}

}

\subsection{Local energy estimate}

In this section, we deliver two local energy estimates, Proposition~\ref{P:EE} for subcritical case and Proposition~\ref{energy cri} for the critical case. Recall that $q_1,q_2$ are given by \eqref{q12}.

\begin{proposition}\label{P:EE}
{ Let $v$ be a non-negative bounded weak solution of \eqref{main v} with $m \geq 1$. Suppose that $\zeta$ is a cutoff function in $Q_{\rho}$ which vanishes on $\partial_{p} Q_{\rho}$ with $0 \leq \zeta \leq 1$. Then  there exists a constant $C= C(m)$ such that for any $k\in(\mu_-,\mu_+)$}
\begin{equation}\label{energy 1}\begin{split}
&\quad \mu_+^{-\beta}\sup_{t_0 \leq t \leq t_1}\int_{K_\rho \times \{t\}} v_+^2\zeta^2\,dx
 +  \iint_{Q_{\rho}} \abs{\nabla (v_+\,\zeta)}^2 \,dx\,dt \\
&\leq 
C(\mu_+-k)^{-\beta}\iint_{Q_{\rho}} v_{+}^{2} {|\zeta\zeta_t|} \,dx\,dt
 +C\iint_{Q_{\rho}} v_{+}^{2} |\nabla\zeta|^2\,dx\,dt \\
&\quad + C\mu_{+}^{2/m} \|B\|^{2}_{L_{t}^{2\hat{q}_1} L_{x}^{2\hat{q}_2} (Q_\rho)} \left[ \int_{t_0}^{t_1} \left[A^{+}_{k,\rho}(t)\right]^{\frac{q_1}{q_2}} \,dt\right]^{\frac{2(1+\kappa)}{q_1}},
\end{split}\end{equation}
and
\begin{equation}\label{energy 2}\begin{split}
&\quad (\mu_-+k)^{-\beta}\sup_{t_0 \leq t \leq t_1}\int_{K_\rho \times \{t\}} v_-^2\zeta^2\,dx
 +  \iint_{Q_{\rho}} |\nabla (v_-\,\zeta)|^2 \,dx\,dt \\
&\leq 
C(\mu_-+k)^{-\beta}k\iint_{Q_{\rho}} v_{-}|\zeta\zeta_t|\,dx\,dt
 +C\iint_{Q_{\rho}} v_{-}^{2}|\nabla\zeta|^2\,dx\,dt \\
&\quad +C(\mu_-+k)^{2/m} \|B\|^{2}_{L_{t}^{2\hat{q}_1} L_{x}^{2\hat{q}_2} (Q_\rho)} \left[ \int_{t_0}^{t_1} \left[A^{-}_{k,\rho}(t)\right]^{\frac{q_1}{q_2}} \,dt\right]^{\frac{2(1+\kappa)}{q_1}}.
\end{split}\end{equation}

\end{proposition}

\begin{proof}

We first consider the set $\{ \mu_{+} - k < v \leq \mu_{+} \}$ which is equivalent to the set $\{ v_{+} > 0\}$. From the equation of $v$ given in \eqref{main v}, we have
\begin{equation}\label{P:EE:01}
0 = \iint_{Q_\rho} \left[(v^{1/m})_t - \Delta v - \nabla (B\cdot v^{1/m})\right] v_+\zeta^2 \,dx\,dt =: I+ II+ III.
\end{equation}

First define
\begin{equation}\label{P:EE:03}\begin{split}
h^{+}(v)
&:= \int_{(\mu_{+} - k)^{\frac{1}{m}}}^{v^{\frac{1}{m}}} \left(\tilde{\sigma}^m - \mu_{+} - k \right)_{+} \,d\tilde{\sigma} \\
&= \frac{1}{m} \int_{\mu_{+} - k}^{v} \left(\sigma - \mu_{+} - k \right)_{+} \sigma^{\frac{1-m}{m}}\,d\sigma
\end{split}\end{equation}
by taking the change of variable. Recall that $\beta = \frac{m-1}{m} \geq 0$. Hence
\[ \mu_{+}^{-\beta} \leq \sigma^{-\beta} \leq (\mu_{+} - k)^{-\beta} \]
which leads upper and lower bounds of $h^{+}(v)$ such that
\begin{equation}\label{P:EE:04}
  \frac{1}{2m \mu_{+}^{\beta}} \, v_{+}^2 \leq h^{+}(v) \leq \frac{1}{2m  (\mu_{+} - k)^{\beta}} \, v_{+}^2.
\end{equation}
Observe that
\begin{equation*}\begin{split}
I &= \iint_{Q_{\rho}} \left(\partial_{t} v^{\frac{1}{m}}\right) v_{+}\zeta^2 \,dx\,dt
  = \iint_{Q_{\rho}} \left( \partial_{t} h^{+}(v) \right) \zeta^2 \,dx\,dt \\
  & = \int_{K_{\rho}\times\{t_1\}} h^{+}(v)  \zeta^2 \,dx - \int_{K_{\rho}\times\{t_0\}} h^{+}(v)  \zeta^2 \,dx - 2\iint_{Q_{\rho}} h^{+}(v)  \zeta \zeta_{t} \,dx\,dt
\end{split}\end{equation*}
It follows from \eqref{P:EE:04} that
\begin{equation}
    \label{P:EE:02}I\geq \frac{1}{2m\mu_+^{\beta}}\int_{K_\rho \times \{t_1\}} v_+^2\zeta^2\,dx-
\frac{1}{m(\mu_+-k)^{\beta}}\iint_{Q_{\rho}} v_{+}^{2} |\zeta\zeta_t| \,dx\,dt.
\end{equation}

Next let us compute the followings
\begin{equation}\label{P:EE:05}
\begin{aligned}
II &= \iint \nabla v \cdot \nabla (v_+\zeta^2) \,dxdt \\
   &=\iint (\nabla (v_+\zeta)-v_+\nabla\zeta) \cdot (\nabla (v_+\zeta)+v_+\nabla\zeta )\,dxdt
    =: II_1+II_2
\end{aligned}
\end{equation}
where
\[
II_1 = \iint_{Q_{\rho}} |\nabla (v_+\,\zeta )|^2  \,dx\,dt,
\quad \text{and} \quad  II_2 =-\iint_{Q_{\rho}} v_+^2 |\nabla \zeta|^2 \,dx\,dt .
\]

Now we handle { the integral quantity that carries the drift term:}
\begin{equation}\label{P:EE:06}
\begin{aligned}
III &= \iint_{Q_\rho} v^{1/m} B \cdot \nabla (v_+\zeta^2) \,dx\,dt\\
&{=\iint_{Q_\rho} v^{1/m} B \cdot \nabla (v_+\zeta)\zeta \,dx\,dt+\iint_{Q_\rho} v^{1/m} B \cdot  (v_+\zeta)\nabla\zeta \,dx\,dt =: III_1 + III_2.}
\end{aligned}
\end{equation}
Then we observe
\[
III_1 \geq -\epsilon II_1 - \frac{1}{4\epsilon}\iint_{Q_\rho \cap \{v_+>0 \}} |B|^2 v^{2/m} \zeta^2 \,dx\,dt
\]
and
\[
III_2 \geq -\epsilon \iint_{Q_{\rho}} v_+^2 |\nabla \zeta|^2 \,dx\,dt
       -\frac{1}{4\epsilon}\iint_{Q_\rho \cap \{ v_+>0 \}} |B|^2 v^{2/m} \zeta^2 \,dx\,dt.
\]

Let us choose $\epsilon = 1/2$. Then the combination of \eqref{P:EE:01}, \eqref{P:EE:02}, \eqref{P:EE:05}, and \eqref{P:EE:06} provide the following
\begin{equation}\label{P:EE:07} \begin{split}
&\quad \frac{1}{2m\mu_+^{\beta}}\int_{K_\rho \times \{t_1\}} v_+^2\zeta^2\,dx
  + \frac{1}{2} \iint_{Q_{\rho}} \abs{\nabla (v_+\,\zeta )}^2  \,dx\,dt\\
&\leq  \frac{1}{m(\mu_+-k)^{\beta}}\iint_{Q_{\rho}} v_{+}^{2} |\zeta\zeta_t| \,dx\,dt
 + \frac{3}{2}\iint_{Q_{\rho}} v_{+}^2 |\nabla \zeta|^2 \,dxdt \\
&\quad + 2\iint_{Q_\rho \cap \{v_{+}>0 \}} |B|^2 v^{2/m} \zeta^2 \,dx\,dt.
\end{split}\end{equation}
Because $v \leq \mu_{+}$, we compute
\[
\iint_{Q_\rho \cap \{v_{+}>0 \}} |B|^2 v^{2/m} \zeta^2 \,dx\,dt
\leq \mu_{+}^{2/m} \|B\|^{2}_{L_{t}^{2\hat{q}_1} L_{x}^{2\hat{q}_2} (Q_\rho)} \left[ \int_{t_0}^{t_1} \left[A^{+}_{k,\rho}(t)\right]^{\frac{q_1}{q_2}} \,dt\right]^{\frac{2(1+\kappa)}{q_1}}
\]
from the setting of $B$ in \eqref{B} with \eqref{hatq12}. Moreover, it is easy to obtain $v_{+} \leq k$. We yield \eqref{energy 1} from \eqref{P:EE:07}.

\vskip 0.2 true in

Secondly, we consider the set $\{ \mu_{-} \leq v \leq \mu_{-} + k \}$ which is equivalent to the set $\{ v_{-} \geq 0\}$. In this case, we consider the following:
\[
0 = \iint_{Q_\rho} \left[ - (v^{1/m})_t + \Delta v + \nabla (B\cdot v^{1/m})\right] v_- \zeta^2 \,dxdt =: I+ II+ III.
\]

Now let us carry
\begin{equation}\begin{split}
I &= \iint_{Q_{\rho}} -\partial_{t}\left( v^{\frac{1}{m}}\right) v_{-}\zeta^2 \,dx\,dt = \iint_{Q_{\rho}} \left(\partial h^{-}(v)\right) \zeta^2 \,dx\,dt\\
&= \int_{K_{\rho}\times\{t_1\}} h^{-}(v) \zeta^2 \,dx - \int_{K_{\rho}\times\{t_0\}} h^{-}(v) \zeta^2 \,dx - 2\iint_{Q_{\rho}} h^{-}(v) \zeta\zeta_t \,dx\,dt
\end{split}\end{equation}
where
\begin{equation}\label{h-lower}\begin{split}
h^{-}(v) &= \int_{v^{\frac{1}{m}}}^{(\mu_{-}+k)^{\frac{1}{m}}} \left( \tilde{\sigma}^m - \mu_{-} - k\right)_{-}\,d\tilde{\sigma} \\
&= \frac{1}{m}\int_{v}^{(\mu_{-}+k)} \left( \sigma - \mu_{-} - k\right)_{-} \sigma^{-\beta}\,d\sigma.
\end{split}\end{equation}

Because $\beta \leq 0$ and $\sigma \leq \mu_{-}+k$, it is easy to have a lower bound of $h^{-}(v)$ such that
\[ h^{-}(v) \geq \frac{1}{2m (\mu_{-}+k)^{\beta}} \, v_{-}^2 . \]
To estimate an upper bound of $h^{-}(v)$, we use of $  \left( \sigma - \mu_{-} - k\right)_{-} \leq k$ so
\begin{equation}\label{h-upper}
h^{-}(v) \leq \frac{k}{m}\int_{v}^{(\mu_{-}+k)} \sigma^{-\beta}\,d\sigma
= k \left( (\mu_{-}+k)^{\frac{1}{m}} - v^{\frac{1}{m}} \right)
\leq k (\mu_{-}+k)^{\frac{1-m}{m}} v_{-}
\end{equation}
because of the following inequality:
\[a-b\leq (a^m-b^m)/a^{m-1}, \quad \text{ for all }\  a> b > 0 . \]

To handle integral terms $II$ and $III$, we follow similar technique above for the upper-level set $v_{+}$. Then we obtain
\begin{equation}\label{P:EE:08} \begin{split}
&\quad \int_{K_{\rho}\times\{t_1\}} h^{-}(v)  \zeta^2 \,dx
  + \frac{1}{2} \iint_{Q_{\rho}} \abs{\nabla (v_-\,\zeta )}^2  \,dx\,dt\\
&\leq  2\iint_{Q_{\rho}} h^{-}(v)  \zeta \zeta_{t} \,dx\,dt
 + \frac{3}{2}\iint_{Q_{\rho}} v_{-}^2 |\nabla \zeta|^2 \,dxdt \\
&\quad + 2\iint_{Q_\rho \cap \{v_{-}>0 \}} |B|^2 v^{2/m} \zeta^2 \,dx\,dt.
\end{split}\end{equation}
Note that $v_{-} \leq k$ and the upper and lower bounds of $h^{-}(v)$ (\eqref{h-upper} and \eqref{h-lower}) lead to \eqref{energy 2} from \eqref{P:EE:08}.

\end{proof}

To prove the uniform continuity of $v$ of \eqref{main v} in the critical range of the drift term in Section~\ref{S:critical}, we derive the following energy estimate:

\begin{proposition}\label{energy cri}
Suppose that $B$ is divergence-free and set $k$ to hold $ \mu_-< k< \mu_+$. Let $\zeta$ be a cutoff function on the parabolic cylinder $Q_\rho$ for $\rho \in (0, \frac{1}{2})$, vanishing on $\partial_p Q_{\rho}$ with $0 \leq \zeta \leq 1$. For a non-negative
bounded weak solution $v$ of \eqref{main v}, there exists a constant $C=C(m)$ such that
\begin{equation}\begin{split}
&\quad \mu_+^{-\beta}\sup_{t_0 \leq t \leq t_1}\int_{K_\rho \times \{t\}} v_+^2\zeta^2\,dx
 +  \iint |\nabla (v_+\,\zeta)|^2 \,dxdt \\
&\leq C\frac{k^{2}}{(\mu_+-k)^\beta}\iint_{v_+>0}|\zeta\zeta_t|\,dx+Ck^2\iint_{v_+>0} |\nabla\zeta|^2\,dxdt\\
&\quad\quad +C(\mu_+-k)^{-2\beta} \|B\|^2_{{L^{2\hat{q}_1}_t L^{ 2\hat{q}_2}_x}(Q_\rho)}
 \left\|v_+\zeta\right\|^2_{L^{q_1}_tL^{q_2}_x(Q_\rho)}.
\end{split}\end{equation}

\end{proposition}

\begin{proof}

We multiply $v_+\zeta^2$ on both sides of \eqref{main v} to get
\begin{align*}
0&=\iint_{Q_\rho}\left[  (v^{1/m})_t  - \Delta v -\nabla (B\cdot v^{1/m})\right] v_+\zeta^2 \,dxdt =: I+II+ III.
\end{align*}
Then the setting of
\[
h^{+}(v) = m^{-1}\int_{0}^{v_+} (\sigma+\mu_{+}-k)^{-\beta}\sigma \, d\sigma
\]
yields that
\[
I = \int h^{+}(v(t_1)) \zeta^2 \,dx-\int h^{+}(v(t_0)) \zeta^2 \,dx - 2 \iint h^{+}(u) \zeta \zeta_t \,dxdt .
\]
{
For $h^{+}(v)$, as before we have
\begin{equation}\label{h plus est}
\frac{1}{2m \mu_+^{\beta}}v_+^2\leq m\, h^+(v)\leq  \frac{1}{2m(\mu_+-k)^\beta}v_+.\end{equation}}
Then
\[I\geq C\mu_+^{-\beta}\int_{K_\rho\times\{t_1\}} v_+^2\zeta^2\,dx-C\frac{k^2}{(\mu_+-k)^\beta}\iint_{v_+>0}|\zeta\zeta_t|\,dxdt.\]

Next let us compute the followings
\begin{align*}
II= \iint \nabla v \cdot \nabla (v_+\zeta^2) \,dxdt=: II_1+II_2
\end{align*}
where
\[
II_1 =\iint |\nabla (v_+\,\zeta)|^2 \,dxdt,\quad II_2=-\iint |v_+\nabla\zeta|^2\,dxdt.
\]

Since $B$ is divergence-free,
\begin{align*}
III &= -\iint \nabla\cdot(B (v^{1/m}-(\mu_+-k)^{1/m} )\,  v_+\zeta^2 \,dxdt \\
&=\iint B (v^{1/m}-(\mu_+-k)^{1/m} )_+ \cdot \nabla( v_+\zeta^2) \,dxdt \\
&=\iint B (v^{1/m}-(\mu_+-k)^{1/m} )_+\zeta \cdot (\nabla( v_+\zeta)+v_+\nabla\zeta) \,dxdt \\
&\geq -\iint |B| (v^{1/m}-(\mu_+-k)^{1/m} )_+\zeta \cdot (|\nabla( v_+\zeta)|+|v_+\nabla\zeta|) \,dxdt.
\end{align*}
We use the simple fact that for any $a>b\geq 0$
\[a-b\leq (a^m-b^m)/ a^{m-1}\]
to obtain
\[(v^{1/m}-(\mu_+-k)^{1/m} )_+\leq (\mu_+-k)^{-\beta}(v-\mu_++k)_+.\]
Then we deduce that
\begin{align*}
III &\geq -(\mu_+-k)^{-\beta}\iint |B|\,v_+\,\zeta\, (|\nabla( v_+\zeta)|+|v_+\nabla\zeta|) \,dxdt\\
&\geq
-\frac{1}{2} II_1 - C(\mu_+-k)^{-2\beta}\iint |B|^2 v_+^2 \zeta^2 \,dxdt+II_2=:\frac{1}{2} II_1 + III_1+II_2.
\end{align*}

For $III_1$, we have
\begin{align*}
    III_1\geq -C(\mu_+-k)^{-2\beta}\iint |B v_+\zeta|^2 \,dxdt
    \geq -C(\mu_+-k)^{-2\beta} \|B\|^2_{{L^{2\hat{q}_1}_t L^{ 2\hat{q}_2}_x}(Q_\rho)}
 \left\|v_+\zeta\right\|^2_{L^{q_1}_tL^{q_2}_x(Q_\rho)}.
\end{align*}
We conclude by putting the above estimates together.

\end{proof}

\subsection{Logarithmic Energy Estimate}

Now we provide a logarithmic energy estimate that is crucial to capture the behavior of a weak solution in terms of the time variable (say the expansion of positivity along the time axis), in particular, Lemma~\ref{L:EPT}. For positive conastants  $\delta$ and $k$ to be adjusted later, denote the function
\begin{equation}\label{Psi}
\Psi(v)=\Psi(v;\delta,k) = \ln^{+} \left[ \frac{k}{(1+\delta)k - (v - \mu_+ + k)_{+}} \right]
\end{equation}
for a non-negative bounded weak solution $v$ of \eqref{main v}. Then it can be observed that
\[
\Psi' (v)= \frac{\partial \Psi}{\partial v} = \frac{1}{(1+\delta)k - (v- \mu_{+} +  k)_{+}} , \ \text{ and } \
\Psi''= (\Psi')^2 .
\]
Trivially $ v \leq \mu_+$, and it follows:
\begin{equation}\label{Psi_bound}
0 \leq \Psi(v)\leq \ln \frac{1}{\delta},\quad \text{ and } \quad  0\leq \Psi'(v)\leq \frac{1}{\delta k}.
\end{equation}

\begin{proposition}\label{P:LE}
For $v$ a non-negative bounded weak solution of \eqref{main v} with $m \geq 1$, let $\Psi(v)=\Psi(v;\delta,k)$ be defined as the above. Suppose that $\zeta=\zeta(x)$ is a cutoff function in $K_{\rho}$ which is vanishing on $\partial K_\rho$ with $0 \leq \zeta \leq 1$. There exists a constant $C=C(m)$ such that for any positive constants $k\leq \mu_+$ and $\delta\leq \frac{1}{2}$, we have
\begin{equation}\label{E:LE}\begin{split}
& \int_{K_\rho \times \{ t_1\}} \Psi^2(v) \zeta^2 \,dx - \int_{K_\rho \times \{ t_0\}} \Psi^2(v) \zeta^2 \,dx
\\
&\leq C {\mu_+^\beta}\iint_{Q_{\rho}}  \Psi(v) |\nabla \zeta|^2 \,dxdt \\
& +C \ln \frac{1}{\delta} \left(\frac{\mu_{+}^{1/m}}{\delta k}+\frac{\mu_{+}^{1+1/m}}{\delta^2 k^{2}}\right) \|B\|^{2}_{L_{t}^{2\hat{q}_1} L_{x}^{2\hat{q}_2} (Q_\rho)} \left[ \int_{t_0}^{t_1} [A^{+}_{k,\rho} (t)]^{\frac{q_1}{q_2}} \,dt \right]^{\frac{2(1+\kappa)}{q_1}}.
\end{split}\end{equation}

\end{proposition}

\begin{proof}
Recall $\beta=\frac{m-1}{m}$. Let us choose the test function
\[
\varphi(x,t) =2m\,v^\beta  \Psi(v) \Psi'(v) \zeta^2.
\]
Then using the property of $\Psi$, note that
\[
\nabla \varphi = 2(m-1) v^{\beta -1} \Psi\Psi' \zeta^2 \nabla v + 2m v^{\beta} (1+\Psi)(\Psi')^2\zeta^2 \nabla v + 4m v^{\beta}\Psi\Psi'\zeta \nabla \zeta.
\]

By testing $\varphi$ to the equation of $v$, \eqref{main v}, we have
\begin{equation}\label{weighted_weak}
0 = \iint_{Q_{\rho}} \left[  (v^{1/m})_t -  \Delta v -  \nabla\cdot (B v^{1/m}) \right] \varphi \,dx\,dt =: I + II + III .
\end{equation}

Then we compute the following integral quantities from \eqref{weighted_weak}:
\[\begin{split}
I &= \iint_{Q_\rho}m\,v^{1-\frac{1}{m}} (v^{\frac{1}{m}})_t \,2 \Psi \Psi' \zeta^2 \,dx\,dt \\
  &= \iint_{Q_\rho} \left(\partial_t \Psi^2 (v)\right) \zeta^2 \,dx\,dt \\
  &= \int_{K_{\rho} \times \{t_1\}} \Psi^2 (v) \zeta^2 \,dx - \int_{K_{\rho}
  \times \{t_0\}} \Psi^2 (v) \zeta^2 \,dx .
\end{split}\]

Moreover, we compute
\[
II = \iint_{Q_\rho} \nabla v \cdot \nabla \varphi \,dx\,dt =: II_1 + II_2 + II_{3},
\]
where
\begin{align*}
II_1 &= \iint_{Q_{\rho}} 2(m-1) v^{-\frac{1}{m}} |\nabla v|^2 \Psi\Psi' \zeta^2 \,dx\,dt, \\
II_2 &= \iint_{Q_{\rho}} 2m v^{\beta} |\nabla v|^2 (1+\Psi)(\Psi')^2\zeta^2 \,dx\,dt, \\
II_3 &= \int_{Q_{\rho}}  4m v^{\beta} \nabla v \cdot \nabla \zeta \Psi\Psi'\zeta\,dx\,dt.
\end{align*}
Note that
\[II_1 \geq 0 , \quad II_2 \geq 0. \]
For $II_3$, we derive
\begin{equation}\label{logII3}
II_3 \geq -\frac{m}{2} \iint_{Q_\rho} v^\beta |\nabla v|^2\Psi(\Psi')^2 \zeta^2 \,dx\,dt
         - 8 m \iint_{Q^0_\rho} v^\beta \Psi |\nabla\zeta|^2 \,dx\,dt.
\end{equation}

To handle the integral quantities involving the drift term, we calculate
\[
III = \iint_{Q_\rho} v^{1/m} B \cdot \nabla \varphi \,dx\,dt =: III_1 + III_2 + III_3
\]
where
\begin{align*}
III_1 &= \iint_{Q_\rho} 2(m-1) B \cdot \nabla v \,\Psi\Psi' \zeta^2 \,dx\,dt, \\
III_2 &= \iint_{Q_\rho} 2m v B \cdot \nabla v \, (1+\Psi)(\Psi')^2\zeta^2  \,dx\,dt, \\
III_3 &= \iint_{Q_\rho} 4m v B \cdot \nabla \zeta \, \Psi\Psi'\zeta \,dx\,dt .
\end{align*}
By taking the Cauchy-Schwartz inequalities, we obtain the followings:
\begin{align*}
III_1 &\geq -\epsilon_1 \iint_{Q_\rho} v^{-\frac{1}{m}} |\nabla v|^2 \Psi \Psi' \zeta \,dx\,dt - (m-1)^2 \epsilon_{1}^{-1} \iint_{Q_\rho} v^{\frac{1}{m}}|B|^2 \Psi\Psi' \zeta^2\,dx\,dt, \\
III_{2} &\geq -\epsilon_{2} \iint_{Q_\rho} v^{\beta} |\nabla v|^2 (1+\Psi)(\Psi')^2 \zeta^2\,dx\,dt \\
       &\quad  - m^2\epsilon_{2}^{-1} \iint_{Q_\rho} |B|^2 v^{1+\frac{1}{m}} (1+\Psi)(\Psi')^2 \zeta^2\,dx\,dt, \\
III_{3} &\geq- \epsilon_{3} \iint_{Q_\rho} v^\beta |\nabla \zeta|^2 \Psi \,dx\,dt
          -4m^2 \epsilon_{3}^{-1} \iint_{Q_{\rho}} |B|^2 v^{1+\frac{1}{m}} \Psi (\Psi')^2 \zeta^2\,dx\,dt.
\end{align*}

The choice of $\epsilon_2 = \frac{m}{2}$ provides
\[
\iint_{Q_{\rho}} m v^{\beta} |\nabla v|^2 |\Psi'(v)|^2\zeta^2 \,dx\,dt
\]
on the left hand side of the logarithmic estimate which can be ignored because of its positivity. Moreover $\epsilon_1 = 2(m-1)$ cancels out $II_1$. Let us fix $\epsilon_3 = m$.

Finally, using the bounds of $v \leq \mu_{+}$ and bounds \eqref{Psi_bound}, we have the following bounds:
\[
{\iint_{Q_\rho} |B|^2 v^{\frac{1}{m}} \Psi\Psi' \zeta^2\,dx\,dt}
\leq \frac{\mu^{\frac{1}{m}}_{+} \ln \frac{1}{\delta}}{\delta k} \|B\|^{2}_{L_{t}^{2\hat{q}_1} L_{x}^{2\hat{q}_2} (Q_\rho)}  \left[ \int_{t_0}^{t_1} [A^{+}_{k,\rho} (t)]^{\frac{q_1}{q_2}} \,dt \right]^{\frac{2(1+\kappa)}{q_1}}
\]
and
\[\begin{split}
& {\iint_{Q_{\rho}} |B|^2 v^{1+\frac{1}{m}}(1+ \Psi) (\Psi')^2 \zeta^2\,dx\,dt} \\
&\leq \frac{(1+\ln \frac{1}{\delta}) \mu_{+}^{1+\frac{1}{m}}}{\delta^2 k^{2}} \|B\|^{2}_{L_{t}^{2\hat{q}_1} L_{x}^{2\hat{q}_2} (Q_\rho)}  \left[ \int_{t_0}^{t_1} [A^{+}_{k,\rho} (t)]^{\frac{q_1}{q_2}} \,dt \right]^{\frac{2(1+\kappa)}{q_1}}.
\end{split}\]
All estimates leads to \eqref{energy 2}.
\end{proof}

\begin{remark}\label{R:intrinsic}(Intrinsic scaling)
From the two energy estimates (Proposition~\ref{P:EE} and Proposition~\ref{P:LE}), we observe a proper intrinsic scaling for $v$, a non-negative bounded weak solution of \eqref{main v}. That is, roughly, to find a proper time length of a local parabolic cylinder in such a way
\begin{equation}\label{intrinsic_v}
(\mu_{\pm} \mp k)^{-\beta} k^2 \zeta_{t} \sim k^2 |\nabla \zeta|^2.
\end{equation}
If assuming $(\mu_{\pm} \mp k)\sim k$, this implies the time interval considered $(t_1-t_0) \sim k^{\beta} \rho^2$
for $k$ a constant chosen proportional to the essential oscillation of $v$ and for $\rho$ the spatial radius.

\end{remark}

\section{Subcritical regime}\label{S:subcritical}

In this section, we establish local H\"{o}lder continuity of $v$, a non-negative bounded weak solution of \eqref{main v} under the condition on $B$ in \eqref{B} satisfying \eqref{hatq12} with $\kappa \in (0, 2/d)$, the \emph{subcritical} region. 

Because we assume the boundedness of $v$, let us denote $k$ as a constant such that
\[
0< k \leq  \|v\|_{L^{\infty}(Q_1)}.
\]
To capture the local behavior of nonlinear diffusion parabolic equation \eqref{main v}, the intrinsic time scaling is essential geometric condition (Remark~\ref{R:intrinsic}). {One may think that a weak solution of degenerate diffusion equations in an intrinsically scaled parabolic cylinder behaves like a weak solution of linear diffusion equations.}
Let us denote the intrinsic time scaling and corresponding parabolic cylinder:
\begin{equation}\label{T}
  T_{k,\rho}(\theta) = \theta k^{-\beta} \rho^2,
\end{equation}
and
\begin{equation}\label{Q}
  Q_{k,\rho}(\theta) = K_{\rho} \times (-T_{k,\rho}(\theta), 0],
\end{equation}
for given constants $k$, $\rho$, and $\theta$. For simplicity, we will denote
\[T_{k,\rho}=T_{k,\rho}(1),\quad  Q_{k,\rho}= Q_{k,\rho}(1).\]

To prove the local H\"{o}lder continuity, we follow the work by DiBenedetto in \cite{DB83} for a general porous medium equation. We modify the proof in \cite{DB83} to cover both degenerate and linear diffusions and to allow weaker conditions on the drifts. We begin by discussing the two alternatives referring a bounded non-negative solution $v$ of \eqref{main v} in an intrinsically scaled cylinder. The first alternative states that, if $v$ is large on most of a cylinder, then $v$ is bounded away from zero on a subcylinder with the same center-top point as the original cylinder.
When the first alternative fails, the second alternative starts from that $v$ is large on a fixed fraction of a cylinder. Then by locating a certain time and applying so-called the expansion of positivities and DeGorgi type iterations (refer \cites{GSV, DBGiVe06, CHKK}), we show that $v$ stays away from its maximum on a subcylinder with the same center-top point as the original cylinder. Eventually, we shall see the precise quantitative description of these results in Section~\ref{SS:Holder}.
We emphasize that our argument holds for the uniform parabolic equation (see Remark~\ref{rmk linear}).

\subsection{The first alternative}\label{SS:first}

For any fixed $\rho\in(0,\frac{1}{8})$, let us denote
\[
\mu_+=\sup_{Q_{k,4\rho}}v,\quad \mu_-=\inf_{Q_{k,4\rho}}v.
\]
Without loss of generality, we assume $ 0\leq \mu_{-} < \mu_{+} < \infty $. Because of the boundeness of $v$, $\mu_{+}=\mu_{-}$ means constant $v$ in $Q_{k, 4\rho}$.
Then we intend to choose a constant $k > 0$ to be proportional to the oscillation of $v$ in $Q_{k, 4\rho}$ assuming
\begin{equation}\label{case1}
  4k\geq \mu_{+} - \mu_{-} \geq 2k.
\end{equation}
When $\mu_-$ is large, the equation is uniformly parabolic which can be treated similarly as the linear equation with $m=1$ i.e. $\beta=0$. Therefore we only consider the two cases that either $\mu_-\leq k$ or $\beta=0$.


The following proposition is the DeGiorgi type iteration as in \cite[Lemma~4.1]{DB83}, \cite[Lemma~3.1]{DB82}.

\begin{proposition}\label{L:FA}
Let $v$ be a bounded non-negative weak solution of \eqref{main v} in $Q_{k, 4\rho}$. Suppose \eqref{case1} and
\begin{equation}
    \label{small mu-}
    \text{ either }\quad  (\mu_-\leq k) \quad \text{ or }\quad  (\beta=0)\quad \text{ holds }.
\end{equation}
Then there exists a number $\nu_0  \in (0,1) $ such that, if
\begin{equation}\label{L:FA:ass}
  \left|\left\{ (x,t) \in Q_{k, 4\rho} : v(x,t) \geq \mu_{-} + k \right\}\right| \leq \nu_0 \left|Q_{k, 4\rho}\right|,
\end{equation}
then either
\begin{equation}\label{L:FA:result1}
  k^{-2-\beta \frac{2(1+\kappa)}{q_1}} \rho^{d\kappa  }>1 ,
\end{equation}
or
\begin{equation}\label{L:FA:result2}
  \essinf_{Q_{k, 2\rho} } v(x,t) > \mu_{-} + \frac{k}{2}.
\end{equation}
\end{proposition}

\begin{proof}
We consider a level set $(v - \mu_{-} - k)_{-}$, equivalently considering the set $\{ \mu_{-} < v < \mu_{-} + k \}$. First, construct sequences $\{\rho_i\}$, $\{k_i\}$, $\{K_{i}\}$, and $\{Q_i\}$ such that
\begin{gather}\label{DGrho}
\rho_{i} = \frac{\rho}{2} + \frac{\rho}{2^{i+1}}, \quad \rho_0 = \rho, \ \rho_{\infty}= \frac{\rho}{2},\\
k_{i} =\frac{k}{2} + \frac{k}{2^{i+1}}, \quad k_{0}=k, \ k_{\infty} =\frac{k}{2},\\
K_{i}= K_{4\rho_i}, \quad Q_{i} =K_{i}\times (- k^{-\beta}
(4\rho_{i})^{2}, 0].
\end{gather}
Moreover, we take a sequence of smooth cutoff functions
$\{ \zeta_{i}\}$ such that
\begin{equation}\label{DGzeta}
\zeta_{i} = \begin{cases}
            1 & \text{ in } Q_{i+1} \\
            0 & \text{ on the parabolic boundary of } Q_{i},
 \end{cases}
\end{equation}
satisfying
\[
|\nabla \zeta_{i}| \leq \frac{2^{i+2}}{4\rho}, \quad \text{ and } \quad
\partial_{t} \zeta_{i} \leq \frac{2^{2(i+2)}}{  k^{-\beta} (4\rho)^{2}}.
\]

It is easy to observe
\[v_{-}^{i}:=(v - \mu_{-}- k_i)_{-} \leq k_i \leq k .\]
The energy estimate \eqref{energy 2} for $v_{-}^{i}$ and $\zeta_i$ in $Q_{i}$ provides the following:
\begin{equation}\label{EE01}\begin{split}
  & (\mu_{-}+k_i)^{-\beta}\sup_{-  k^{-\beta}(4\rho_{i})^2 < t < 0} \, \int_{K_{i}\times\{t\}} ( v_-^{i})^2 \zeta^2_i \,dx
  +  \iint_{Q_{i}} \, |\nabla (v_-^{i}\zeta_i )|^2 \,dx\,dt \\
  &\leq C (\mu_{-}+k_i)^{-\beta} k_i \iint_{Q_i} v^{i}_{-} \zeta_i \partial_{t} \zeta_i \,dx\,dt + C \iint_{Q_i} (v^i_{-})^2 |\nabla \zeta_i|^2 \,dx\,dt \\
  &\quad + C (\mu_-+k)^{2/m} \|B\|^{2}_{L_{t}^{2\hat{q}_1} L_{x}^{2\hat{q}_2}}
  \left[ \int_{-  k^{-\beta}(4\rho_{i})^2}^{0} \left[A_{-}^{i}(t)\right]^{\frac{q_1}{q_2}} \,dt\right]^{\frac{2(1+\kappa)}{q_1}}
\end{split}\end{equation}
by denoting $A_{-}^{i} = A_{k_i, 4\rho_i}^{-}$.

We observe that
\[v_{-}^{i}:=(v - \mu_{-}- k_i)_{-} \leq k_i \leq k \]
and
\begin{equation*}
 (k/2)^{\beta}\leq (\mu_{-} + k_i)^{\beta} \leq (\mu_{-} + k)^{\beta}
\end{equation*}
because
\[
k/2 \leq \mu_{-} + k/2 \leq \mu_{-}+ k_i \leq \mu_{-}+ k .
\]
Multiplying $(\mu_{-} + k_i)^{\beta}$ to \eqref{EE01} leads the following:
\begin{equation}\label{EE02}\begin{split}
  & \sup_{- k^{-\beta} (4\rho_{i})^2 < t < 0} \, \int_{K_{i}\times\{t\}} ( v_-^{i})^2 \zeta^2_i \,dx
  +  (k/2)^{\beta}\iint_{Q_{i}} \, |\nabla (v_-^{i}\zeta_i )|^2 \,dx\,dt \\
  &\leq C \left[ \frac{2^{2i}k^2}{ k^{-\beta} \rho^2} + \frac{(\mu_{-}+k)^{\beta}k^2 2^{2i}}{\rho^2}\right] \left|A_{-}^{i}\right| \\
  &\quad + C (\mu_-+k)^{1+1/m} \|B\|^{2}_{L_{t}^{2\hat{q}_1} L_{x}^{2\hat{q}_2}}
  \left[ \int_{-  k\rho_{i}^2}^{0} \left[A_{-}^{i} (t)\right]^{\frac{q_1}{q_2}} \,dt\right]^{\frac{2(1+\kappa)}{q_1}}
\end{split}\end{equation}
using the bounds and properties of $\zeta_i$.

Now we take the change of variable
\[
\bar{t} = k^{\beta}t \in [- (4 \rho_{i})^2, 0].
\]
Also denote $\bar{v} = v(\cdot, \bar{t})$, $\bar{\zeta}_{i}=\zeta_i (\cdot, \bar{t})$, and $\bar{Q}_i = K_{i} \times (- (4\rho_i)^2, 0]$. Moreover name the sets $\bar{A}^{-}_{k_i, \rho_i}$ and $\bar{A}^{-}_{k_i, \rho_i}$ corresponding to $A^{-}_{k_i, \rho_i}$ and $A^{-}_{k_i, \rho_i}(t)$ with $\bar{t}$, respectively.
From \eqref{EE02}, we derive that
\begin{equation*}
\begin{split}
   & \sup_{-  \rho_{i}^2 < \bar{t} < 0} \, \int_{K_{i}\times\{\bar{t}\}} ( \bar{v}_-^{i})^2 \bar{\zeta}^2_i \,dx
  +  C \iint_{\bar{Q}_{i}} \, |\nabla (\bar{v}_-^{i}\bar{\zeta}_i )|^2 \,dx\,d\bar{t} \\
  &\leq C \left[ \frac{2^{2i}k^2}{ \rho^2} + \left(\frac{\mu_{-}+k}{k}\right)^{\beta}\frac{k^2 2^{2i}}{\rho^2}\right] \left|\bar{A}_{-}^{i}\right| \\
  &\quad + C \left({\mu_-+k}\right)^{1+1/m} k^{-\beta \frac{2(1+\kappa)}{q_1}} \|B\|^{2}_{L_{t}^{2\hat{q}_1} L_{x}^{2\hat{q}_2}}
  \left[ \int_{-  \rho_{i}^2}^{0} \left[\bar{A}_{-}^{i} (t)\right]^{\frac{q_1}{q_2}} \,d\bar{t}\right]^{\frac{2(1+\kappa)}{q_1}}.
\end{split}\end{equation*}
{In either the case $\beta=0$ or $\mu_-\leq k$, we can further simplify the right hand side by}
\begin{equation}\label{EE03}
\begin{split}
   & \sup_{-  \rho_{i}^2 < \bar{t} < 0} \, \int_{K_{i}\times\{\bar{t}\}} ( \bar{v}_-^{i})^2 \bar{\zeta}^2_i \,dx
  +  C \iint_{\bar{Q}_{i}} \, |\nabla (\bar{v}_-^{i}\bar{\zeta}_i )|^2 \,dx\,d\bar{t} \\
  &\leq C \left[ \frac{2^{2i}k^2}{ \rho^2} + \frac{k^2 2^{2i}}{\rho^2}\right] \left|\bar{A}_{-}^{i}\right| \\
  &\quad + C  k^{-\beta \frac{2(1+\kappa)}{q_1}} \|B\|^{2}_{L_{t}^{2\hat{q}_1} L_{x}^{2\hat{q}_2}}
  \left[ \int_{-  \rho_{i}^2}^{0} \left[\bar{A}_{-}^{i} (t)\right]^{\frac{q_1}{q_2}} \,d\bar{t}\right]^{\frac{2(1+\kappa)}{q_1}}.
\end{split}\end{equation}

The Sobolev embedding (Theorem~\ref{T:Sobolev}) for $\bar{v}_{-}^{i}\bar{\zeta}_i$ and the property of $\zeta_i$ implies the following:
\begin{equation}\label{EE04}\begin{split}
& \iint_{\bar{Q}_{i+1}}  (\bar{v}- \mu_{-} - k_{i})_{-}^{2}  \,dx \,d\bar{t}
\leq \iint_{\bar{Q}_{i}}  (\bar{v}- \mu_{-} - k_{i})_{-}^{2}\bar{\zeta}_{i}^{2} \,dx \,d\bar{t} \\
&\leq |\bar{A}_{-}^{i}|^{\frac{2}{d+2}}\left[
\sup_{ -  (4\rho_{i})^{2} < \bar{t} < 0} \, \int_{K_{i}\times\{\bar{t}\}} (\bar{v}_{-}^{i})^2 \bar{\zeta}^2_i \,dx
 + C \iint_{\bar{Q}_{i}} \, \left|\nabla (\bar{v}_{-}^{i} \bar{\zeta}_i)\right|^2  \,dx\,d\bar{t} \right].
\end{split}\end{equation}

In the set $ \{ (\bar{v} - \mu_{-} -k_{i+1})_{+} >0 \}$ equivalently $\{ \bar{v} < \mu_{-} + k_{i+1}\}$, we observe that
\begin{equation}\label{EE05}
(\bar{v}- \mu_{-} - k_{i})_{-} \geq k_{i} - k_{i+1} = \frac{k}{2^{i+2}}.
\end{equation}
It follows by combining \eqref{EE03}, \eqref{EE04}, and \eqref{EE05} and by multiplying $k^2 /2^{2i}$:
\begin{equation}\label{EE06}\begin{split}
 & |\bar{A}_{-}^{i+1}| \leq   C \frac{2^{4i}}{\rho^2} \left|\bar{A}_{-}^{i}\right|^{1+\frac{2}{d+2}}  \\
  &+ C 2^{2i}\|B\|^{2}_{L_{t}^{2\hat{q}_1} L_{x}^{2\hat{q}_2}}
   k^{-2-\beta \left(\frac{2 (1+\kappa)}{q_1}\right)}
  \left[\int_{-  (4\rho_{i})^2}^{0} \left[\bar{A}_{-}^{i}(\bar{t})\right]^{\frac{q_1}{q_2}} \,d\bar{t}\right]^{\frac{2(1+\kappa)}{q_1}}
  |\bar{A}^{-}_{i}|^{\frac{2}{d+2}}.
\end{split}\end{equation}

We note that
\[
|\bar{Q}_i| \sim |\bar{Q}_{i+1}| \sim  \rho^{d+2}, \quad |\bar{Q}_i|^{\frac{2}{2+d}} \sim  \rho^2.
\]

Let us denote
\[
\bar{Y}_{i} = \frac{|\bar{A}_{-}^{i}|}{|\bar{Q}_i|}, \quad \bar{Z}_i = \frac{1}{|K_i|} \left[ \int_{-  (4\rho_{i})^2}^{0} \left[\bar{A}_{-}^{i}(\bar{t})\right]^{\frac{q_1}{q_2}} \,d\bar{t}\right]^{\frac{2}{q_1}}.
\]

Then \eqref{EE06} becomes
\begin{equation}\begin{split}
 & \bar{Y}_{i+1} \leq   C  2^{4i} \bar{Y}_{i}^{1+ \frac{2}{d+2}}  \\
  &+ C  2^{2i}\|B\|^{2}_{L_{t}^{2\hat{q}_1} L_{x}^{2\hat{q}_2}}
   k^{-2-\beta \left(\frac{2 (1+\kappa)}{q_1}\right)}\rho^{d\kappa} \bar{Y}_{i}^{\frac{2}{d+2}} \bar{Z}_{i}^{1+\kappa}.
\end{split}\end{equation}

If
\begin{equation}\label{EE07}
  k^{-2-\beta \left(\frac{2 (1+\kappa)}{q_1}\right)}\rho^{d\kappa} > 1,
\end{equation}
then this implies \eqref{L:FA:result1} (H\"{o}lder continuity directly). Otherwise, when \eqref{EE07} fails, we have the following inequality:
\begin{equation}\label{EE08}
\bar{Y}_{i+1} \leq   C  2^{4i} \bar{Y}_{i}^{1+ \frac{2}{d+2}}
  + C 2^{2i}\|B\|^{2}_{L_{t}^{2\hat{q}_1} L_{x}^{2\hat{q}_2}} \bar{Y}_{i}^{\frac{2}{d+2}} \bar{Z}_{i}^{1+\kappa}.
\end{equation}

Since \eqref{EE08} is the dimensionless form, we take the change of variable to $t$ from $\bar{t}$ which results
\begin{equation}\label{EE010}
Y_{i+1} \leq   C  2^{4i} Y_{i}^{1+ \frac{2}{d+2}}
  + C 2^{2i}\|B\|^{2}_{L_{t}^{2\hat{q}_1} L_{x}^{2\hat{q}_2}} Y_{i}^{\frac{2}{d+2}} Z_{i}^{1+\kappa}.
\end{equation}

From the setting of $\bar{Z}_{i}$, we observe that
\begin{equation}\label{EE071} \begin{split}
& (k_{i} - k_{i+1})^2 \left[ \int_{-  (4\rho_{i+1})^2}^{0} \left[\bar{A}_{-}^{i+1}(\bar{t})\right]^{\frac{q_1}{q_2}} \,d\bar{t}\right]^{\frac{2}{q_1}} \\
&\leq \left[ \int_{- k^{-\beta}(4\rho_{i+1})^2}^{0}
\left[ \int_{K_{i+1}} (\bar{v}-\mu_{-} - k_{i})_{-}^{q_2} \,dx\right]^{\frac{q_1}{q_2}} \,d\bar{t}\right]^{\frac{2}{q_1}} \\
&\leq \left[ \int_{-  k^{-\beta}(4\rho_{i})^2}^{0}
\left[ \int_{K_{i}} (\bar{v}-\mu_{-} - k_{i})_{-}^{q_2} \bar{\zeta}_{i}^{q_2} \,dx\right]^{\frac{q_1}{q_2}} \,d\bar{t}\right]^{\frac{2}{q_1}}\\
& = \|\bar{v}_{-}^{i} \zeta_{i}\|_{L_{t}^{q_1}L_{x}^{q_2} (Q_i)}^{2}
\end{split}\end{equation}
first because of $ (\bar{v}-\mu_{-} - k_{i})_{-} \geq k_i - k_{i+1}$ in the set $\{ \bar{v} < \mu_{-}+k_{i+1}\}$ and second using the property of cutoff function $\bar{\zeta}_i$.
On the other side, we apply the embedding given in Proposition~\ref{P:admissible} and then use \eqref{EE02} to have
\begin{equation}\label{EE081}
\|\bar{v}_{-}^{i} \bar{\zeta}_{i}\|_{L_{t}^{q_1}L_{x}^{q_2} (Q_i)}^{2} \leq C\|\bar{v}_{-}^{i} \bar{\zeta}_{i}\|_{V^{2}(\bar{Q}_{i})}^{2}
\leq C \frac{k^2}{\rho^2} 2^{2i}  \left|\bar{A}_{-}^{i}\right| + C k^2 \rho^{d} \bar{Z}_{i}^{1+\kappa}.
\end{equation}
Then the combination of \eqref{EE071} and \eqref{EE081} provides
\[
\bar{Z}_{i+1} \leq C2^{4i} \frac{|\bar{A}_{-}^{i}|}{|\bar{Q}_i|} + C 2^{2i} \bar{Z}_{i}^{1+\kappa}
\]
which is dimensionless form. Therefore the change of variable to $t$ gives
\begin{equation}\label{EE9}
Z_{i+1} \leq C 2^{4i} Y_{i} + C 2^{2i} Z_{i}^{1+\kappa}.
\end{equation}

With \eqref{EE010} and \eqref{EE9}, we apply Lemma~\ref{L:iter2} to conclude that both $Y_i$ and $Z_i$ tend to zero as $i \to \infty$, provided
\[
Y_0 + Z_{0}^{1+\kappa} \leq (2C)^{-\frac{1+\kappa}{\sigma}} 2^{-\frac{4(1+\kappa)}{\sigma^2}} := \nu_0
\]
where $\sigma = \min\{\kappa, \frac{2}{d+2}\}$. Because
\[
Y_{\infty}= \frac{1}{|Q_{k, 4\rho}|}|\{(x,t)\in K_{2\rho}\times (-k^{-\beta}(2\rho)^2,0]\}|   = 0,
\]
this directly gives our conclusion \eqref{L:FA:result2}.
\end{proof}

\subsection{The second alternative}\label{SS:SA}

Now suppose that the assumption for the first alternative \eqref{L:FA:ass} fails that gives
\begin{equation}\label{SA:ass}
  \left|\left\{ (x,t) \in Q_{k,4\rho}: v(x,t) \geq \mu_{-}+k \right\}\right| \leq (1-\nu_0) \left|Q_{k,4\rho}\right|.
\end{equation}
By \eqref{case1}, this leads to
\begin{equation}\label{L1:ass}
 \left|\left\{ Q_{k,4\rho}: v(x,t) \geq \mu_{+} - k \right\}\right| \leq (1-\nu_0) \left|Q_{k,4\rho}\right|,
\end{equation}
The second alternative starts from \eqref{L1:ass} which is basically some measure control overall the parabolic cube $Q_{k, 4\rho} $. Our goal is to locate a smaller parabolic cylinder sharing the top vertex $(0,0)$ where the essential supremum becomes strictly less than $\mu^{+}$. We need to follow quiet a few steps the combinations of so-called expansion of positivity and DeGiorgi iteration.

The following lemma says that if a non-negative function is large on some part of a cylinder, then it keeps largeness on part of a suitable time slice same as \cite[Lemma~4.2]{DB83}.
\begin{lemma}\label{L1}
{Let $\nu_0 \in (0,1)$ be from Proposition \ref{L:FA} and $\{k,\mu_+,\mu_-,\rho\}$ satisfy \eqref{case1} and \eqref{small mu-}.} If \eqref{L1:ass} holds,
then there exists a time level
\begin{equation}\label{L1:tau}
-t_0 \in (- k^{-\beta} (4\rho)^2, - \frac{\nu_0}{2}  k^{-\beta} (4\rho)^2)
\end{equation}
such that
\begin{equation}\label{L1:res}
  \left|\left\{ K_{4\rho}: v(x, -t_0) \geq \mu_{+} - k \right\}\right| \leq \frac{1-\nu_0}{1-\nu_0 /2} \left|K_{4\rho}\right|.
\end{equation}
\end{lemma}

\begin{proof}
Suppose not. Then \eqref{L1:res} fails for all $t_0$ in \eqref{L1:tau} which implies
\begin{align*}
 & \left|\left\{ Q_{k,4\rho}: v(x,t) \geq \mu_{+} - k \right\}\right| \\
 &\geq \int_{- k^{-\beta}(4\rho)^2}^{- \frac{\nu_0}{2}  k^{-\beta}(4\rho)^2} \left| \left\{K_{4\rho}: v(x,t) \geq \mu_{+} - k \right\} \right| \,dt \\
 &> (1-\nu_0) \left|Q_{k,4\rho}\right|
\end{align*}
that contradicts with \eqref{L1:ass}.
\end{proof}

{
Now from the time level $-t_0$ from \eqref{L1:res}, we are able to control the measure where a weak solution keeps its largeness in a certain later time using the logarithmic energy estimate in Proposition~\ref{P:LE}.
}
\begin{lemma}\label{L:EPT}(Expansion of positivity in time)
{Under the conditions of Lemma \ref{L1},}
then, there exists $\delta_1 = \delta_1 (\nu_0)<\frac{1}{2}$
such that, if
\begin{equation}\label{L:EPT:ass}
\left| \{ x\in K_{4\rho}: v(x,-t_0) > \mu_{+} -k \}\right| \leq  \frac{1-\nu_0}{1-\nu_0 /2} |K_{4\rho}|
\end{equation}
for some
$-t_0 \in(  -k^{-\beta} (4\rho)^2,-\frac{\nu_0}{2}k^{-\beta}(4\rho)^2)$,
then either
\begin{equation}\label{L:EPT:kr}
 k^{-2-\beta \left(  \frac{2(1+\kappa)}{q_1}\right)} \rho^{d\kappa} \geq \delta_1^2,
\end{equation}
or
\begin{equation}\label{L:EPT:res}
\left| \{ x\in K_{4\rho} : v(x,t) > \mu_{+} - \delta_1 k \} \right| < \left(1-\frac{\nu_0^2}{4} \right)|K_{4\rho}|
\end{equation}
for all $-t \in (-t_0, 0]$.
\end{lemma}

\begin{proof}
Write $r=4\rho$ for simplicity.
We apply the logarithmic energy estimates from Proposition~\ref{P:LE}. For some $\lambda\in(0,1)$ to be determined later, denote that $\zeta$ is a linear cutoff function independent of the time variable, such that
\begin{equation*}
\zeta = \begin{cases}
 1 &\text{ in } K_{(1-\lambda)r} \times  [ -t_0, -t], \\
 0 &\text{ on the lateral boundary of } K_{r} \times  [ -t_0, -t].
 \end{cases}
\end{equation*}
for any $-t\in (-t_0, 0]$,  satisfying
\[
0\leq \zeta \leq 1, \quad |\nabla \zeta| \leq \frac{1}{\lambda r}, \quad \zeta_t = 0.
\]

Let $\delta_1 = 2^{-j}$ for $j$ a positive integer to be chosen later large enough. From the setting of $\Psi(v)$ in \eqref{Psi}, we have
\[
\Psi(v) \leq \ln \frac{1}{\delta_1} = j \ln 2, \quad \text{and} \quad \Psi'(v) \leq \frac{1}{\delta_1 k}.
\]
Moreover, in the set $\{ v > \mu_{+} -\delta_1 k \}$, it is obvious to have
\[
(v - \mu_{+} + k)_{+} > (1-\delta_1) k
\]
which gives
\[
\Psi(v) \geq \ln \frac{1}{2\delta_1} = (j-1) \ln 2.
\]

Now, we start from the logarithmic estimate \eqref{E:LE}. The various estimates of $\Psi(v)$, $\Psi'(v)$, the properties of $\zeta$, and \eqref{L:EPT:ass} yield the following inequality:
\begin{equation}\label{EPT01}\begin{split}
 & (j-1)^2 (\ln 2)^2 \left| \{ x\in K_{(1-\lambda)r}: v(x,t) > \mu_{+} - \delta_1 k\} \right| \\
&\leq j^2 (\ln 2)^2 \left(\frac{1-\nu_0}{1-\nu_0/2}\right)|K_r|
 + \frac{C (j \ln 2)}{\lambda^2 r^2} \mu_+^\beta|K_r| |t-t_0| \\
& +C (j\ln 2) \|B\|^{2}_{L_{t}^{2\hat{q}_1} L_{x}^{2\hat{q}_2}}\left(\frac{\mu_{+}^{1/m}}{\delta_1 k}+\frac{\mu_{+}^{1+1/m}}{\delta_1^2 k^{2}}\right)  |K_{r}|^{\frac{2(1+\kappa)}{q_2}} |t-t_0|^{\frac{2(1+\kappa)}{q_1}}.
\end{split}
\end{equation}

{Note that $t-t_0\leq k^{-\beta}r^2$. If $\mu_-\leq k$, we know $k$ is chosen to be proportional to $\mu_+$. } So
in the case either $\beta=0$ or $\mu_-\leq \frac{\omega}{4}$, we simplify \eqref{EPT01} to the following:
\begin{equation}\label{EPT02}\begin{split}
 & (j-1)^2 (\ln 2)^2 \left| \{ x\in K_{(1-\lambda)r}: v(x,t) > \mu_{+} - \delta_1 k\} \right| \\
&\leq j^2 (\ln 2)^2 \left(\frac{1-\nu_0}{1-\nu_0/2}\right)|K_r|
 + \frac{C  (j \ln 2) }{\lambda^2 } |K_r|  \\
& +C (j\ln 2) \|B\|^{2}_{L_{t}^{2\hat{q}_1} L_{x}^{2\hat{q}_2}}
\left(\frac{\mu_+^{1+\frac{1}{m}}}{\delta_1^2}\right) k^{-2 - \frac{2\beta(1+\kappa)}{q_1}} r^{d\kappa} |K_r|.
\end{split}
\end{equation}

When \eqref{L:EPT:kr} holds, then it implies \eqref{L:FA:result1} which is H\"{o}lder continuous. Let us assume that
\begin{equation}\label{kr1}
   k^{-2 - \frac{2\beta(1+\kappa)}{q_1}} r^{d\kappa} \leq \delta_1^2 .
\end{equation}

Therefore, the following is from \eqref{EPT02} applying \eqref{kr1} and \eqref{L:EPT:ass}:
\begin{equation}\label{EPT03}\begin{split}
 & (j-1)^2 (\ln 2)^2 \left| \{ x\in K_{(1-\lambda) r}: v(x,t) > \mu_{+} - \delta_1 k\} \right| \\
&\leq j^2 (\ln 2)^2 \left(\frac{1-\nu_0}{1-\nu_0/2}\right)|K_r|
 + \frac{C  (j \ln 2) }{\lambda^2 } |K_r|  + C j\ln 2 |K_r|.
\end{split}
\end{equation}
By dividing \eqref{EPT03} by $(j-1)^2 (\ln 2)^2$ and considering the geometric property of the set $ |K_r \setminus K_{(1-\lambda)r}|\leq d\lambda$, it follows:
\begin{equation}\label{claim2:04}\begin{split}
 & \left| \{ x\in K_r: v(x,t) > \mu_{+} - \delta_1 k\} \right| \\
&\leq \left[ \frac{j^2}{(j-1)^2} (\frac{1-\nu_0}{1-\nu_0 / 2})
 + \frac{C  j }{ (j-1)^2\lambda^2 }  +  \frac{C j}{(j-1)^2 }  + d\lambda\right]|K_r|.
\end{split}
\end{equation}

Then we make a choice of $j$ large enough such that
\[
\frac{j^2}{(j-1)^2} \leq (1-\frac{\nu_0}{2})(1+\nu_0),
\]
and
\[ \frac{C  (j \ln 2) }{ (j-1)^2\lambda^2 }\leq \frac{ \nu_0^2}{4},\ \frac{C j}{(j-1)^2 \ln 2}\leq \frac{ \nu_0^2}{4}, \ d\lambda \leq \frac{ \nu_0^2}{4}.
\]
The above is true with the choice of $\lambda = \frac{ \nu_0^2}{4d}$ and then $j$ large enough depending only on $\nu_0,d$. This completes the proof.


\end{proof}

The conclusion of Lemma~\ref{L:EPT} says that the portion where $v$ is greater than some constant is controlled at each time slice. To obtain eventual pointwise estimates of $v$, we need to take the following step (so-called the expansion of positivity in space) that implies a certain level for having arbitrary control by any $\nu_1 \in (0,1)$ over the set where $v$ is bigger than the level based on the measure control at each time \eqref{L:EPT:res}. The following lemma is parallel to \cite[Lemma~4.4, Corolloary~4.5, Lemma~4.6]{DB83}, \cite[Proposition~4.7]{CHKK} and \cite[Section~4]{DBGiVe06}.



\begin{lemma}\label{L:EPX}(Expansion of positivity in space)
Under the conditions of Lemma \ref{L:EPT}.
For any $\nu_1\in(0,1)$, there exists $\delta^\ast = \delta^\ast (\nu_1)<\delta_1$ such that either
\begin{equation}\label{L:EPX:kl}
  k^{-2 - \frac{2\beta(1+\kappa)}{q_1}} \rho^{d\kappa} > {\delta^\ast}^2,
\end{equation}
or
\begin{equation}\label{L:EPX:res}
\left| (x,t)\in Q_{k,2\rho}({\nu_0}): v > \mu_{+} - \delta^\ast k \} \right| \leq \nu_1 \left|Q_{k,2\rho}({\nu_0})\right|.
\end{equation}
\end{lemma}

\begin{proof}
Write $r=2\rho$ and $t_1:={\nu_0}k^{-\beta}r^2 $. By definition, we have $2t_1\leq t_0$. It follows from Lemma \ref{L:EPT} that if \eqref{L:EPX:kl} fails, then
\begin{equation}\label{L:EPX:ass}
\left| \{x\in K_{2r}: v(x,t) > \mu_{+}-\delta_1k \} \right| \leq (1-\frac{\nu_0^2}{4})|K_{2r}|
\end{equation}
for all $t\in [-2t_1, 0]$.

Let $j = \ceil*{\log_2\frac{1}{\delta_1}}, \ldots, j^{*}$ where $j^{*}$ is to be determined later.
For each such $j$, let $k_{j} = 2^{-j}k$ and so $k_j\leq \delta_1k$.  Denote $\delta^\ast = 2^{-j^\ast}$, so $k_{j^\ast}=\delta^\ast k$.
For simplicity, let us name
\begin{gather*}
v_{+}^{j} = (v - \mu_{+} + k_{j})_{+}, \\
A_{j} (t) = \{ x\in K_{r} : \, v(x,t) > \mu_{+} - k_{j}\}, \\
A_{j} = \{ (x,t)\in K_{r}\times [-t_1, 0] : \, v(x,t) > \mu_{+} - k_{j}\}.
\end{gather*}

Let $\zeta$ be a piecewise linear cutoff function
\[
\zeta = \begin{cases}
 1 &\text{ in } K_{r}\times [-t_1, 0], \\
 0 & \text{ on the parabolic boundary of } K_{2r}\times [-2 t_1, 0],
 \end{cases}
\]
satisfying
\[
0\leq \zeta \leq 1, \quad |\nabla \zeta|\leq \frac{2}{r}, \quad
\zeta_{t} \leq \frac{1}{t_1} = \frac{2k^\beta}{\nu_0r^2}.
\]

It is trivial that
\[
 v_{+}^{j} \leq k_{j} \leq k, \quad \text{ and } \quad  \mu_{+} - k \leq \mu_{+} - k_{j} \leq v \leq \mu_{+}.
\]
Then the local energy estimate \eqref{energy 1} from Proposition~\ref{P:EE}, ignoring the supremum term on the left hand side because of its positivity, leads to the following inequality:
\begin{equation}\label{EPX01}\begin{split}
&\quad \int_{-2t_1}^{0}\int_{K_{2r}} \abs{\nabla (v_+^{j}\,\zeta)}^2 \,dx\,dt \\
&\leq C(\mu_+-k)^{-\beta} \int_{-2t_1}^{0}\int_{K_{2r}} (v_{+}^j)^2|\zeta\zeta_t| \,dx\,dt
 +C \int_{-2t_1}^{0}\int_{K_{2r}} (v_{+}^j)^2 |\nabla\zeta|^2\,dx\,dt \\
&\quad + C\mu_{+}^{2/m} \|B\|^{2}_{L_{t}^{2\hat{q}_1} L_{x}^{2\hat{q}_2}} \left[ \int_{-2t_1}^{0} \left| K_{2r} \cap \{ v(x,t) > \mu_{+} - k_j \} \right|^{\frac{q_1}{q_2}} \,dt\right]^{\frac{2(1+\kappa)}{q_1}}.
\end{split}\end{equation}

By applying properties of $\zeta$, the inequality \eqref{EPX01} becomes
\begin{equation}\label{EPX02}\begin{split}
&\int_{-2t_1}^{0}\int_{K_{2r}} \abs{\nabla (v_+^{j}\,\zeta)}^2 \,dx\,dt \\
&\leq {C} \left(\frac{k}{\mu_+-k}\right)^{\beta} \frac{k_j^2}{\nu_0 r^2} |K_r \times [-2t_1, 0]|+ C \frac{k_j^2}{r^2} |K_r \times [-2t_1, 0]| + C\|B\|^{2}_{L_{t}^{2\hat{q}_1} L_{x}^{2\hat{q}_2}}\mathcal{I}.
\end{split}\end{equation}
We note that in the case either $\mu_-\leq k$ or $\beta=0$, $ (\frac{k}{\mu^+ - k})^\beta \leq  C$ for some universal $C$.
Also we observe that
\[\begin{split}
\mathcal{I} &= \mu_{+}^{2/m} \left[ \int_{-2t_1}^{0} \left| K_{2r} \cap \{ v(x,t) > \mu_{+} - k_j \} \right|^{\frac{q_1}{q_2}} \,dt\right]^{\frac{2(1+\kappa)}{q_1}} \\
&\leq \mu_{+}^{2/m} |K_{2r}|^{\frac{2(1+\kappa)}{q_2}} (2t_1)^{\frac{2(1+\kappa)}{q_1}} \\
&\leq  2^{2j} k^{-2-\beta  \frac{2(1+\kappa)}{q_1} } r^{d\kappa} \frac{k_j^2}{r^2} |K_r \times [-2t_1, 0]|
\end{split}\]
because $k$ is proportional to $\omega$.

Let
\begin{equation}\label{kr2}
  k^{-2-\beta \frac{2(1+\kappa)}{q_1} } r^{d\kappa} < {\delta^\ast}^2.
\end{equation}
Otherwise, we have \eqref{L:EPX:kl} which later yields \eqref{L:FA:result1}.
We have a bound of $\mathcal{I}$ which gives
\begin{equation}\label{EPX03}
\int_{-2t_1}^{0}\int_{K_{2r}} \abs{\nabla (v_+^{j}\,\zeta)}^2 \,dx\,dt
\leq C \frac{k_j^2}{\nu_0 r^2} |K_r \times [-2t_1, 0]|.
\end{equation}

Now we apply Lemma~\ref{Iso} with $k = \mu_{+} -k_j$ and $l = \mu_{+} - k_{j+1}$ which provides
\begin{equation}\label{EPX04}
\begin{aligned}
(k_{j}-k_{j+1}) |A_{j+1}(t)|& \leq \frac{\gamma r^{d+1}}{  | K_{r}\cap\{ v\leq \mu_+- k_j \}|} \int_{A_{j}(t)\setminus A_{j+1}(t)} |\nabla v_{+}^{j}| \,dx\\
&\leq  \frac{\gamma r^{d+1}}{ \alpha | K_{r}|} \int_{A_{j}(t)\setminus A_{j+1}(t)} |\nabla v_{+}^{j}| \,dx
\end{aligned}
\end{equation}
where $\gamma$ is a universal constant and $\alpha:=\frac{\nu_0^2}{4}$. In the second inequality, we used \eqref{L:EPX:ass} and the fact that $k_j<\delta_1k$.

Next by taking integration of \eqref{EPX04} over the time variable $t\in [-t_1, 0]$ and by dividing the result inequality by $|A_{j}\setminus A_{j+1}|$, it yields
\begin{equation}\label{EPX05}\begin{split}
\frac{k}{2^{j+1}}\frac{|A_{j+1}|}{|A_{j}\setminus A_{j+1}|}
&\leq \frac{\gamma r}{ \alpha} \frac{1}{|A_{j}\setminus A_{j+1}|}\iint_{A_{j}\setminus A_{j+1}} |\nabla v_{+}^{j}| \,dx\,dt \\
&\leq \frac{\gamma r}{\nu_0 \alpha} \left(\frac{1}{|A_{j}\setminus A_{j+1}|}\iint_{A_{j}\setminus A_{j+1}} |\nabla v_{+}^{j}|^{2} \,dx\,dt \right)^{\frac{1}{2}}.
\end{split}\end{equation}

Then we square \eqref{EPX05} and rearrange the inequality to obtain
\begin{equation}\label{EPX06} \begin{split}
|A_{j+1}|^2
&\leq \frac{\gamma^2 r^2}{\alpha^2 k_j^2}|A_{j}\setminus A_{j+1}|\iint_{A_{j}\setminus A_{j+1}} |\nabla v_{+}^{j}|^{2} \,dx\,dt \\
&\leq \frac{4\gamma^2}{\alpha^2}  |A_{j}\setminus A_{j+1}| |K_{r}\times [-2t_1, 0]|
\end{split}\end{equation}
because of \eqref{EPX03}.

Now we divide \eqref{EPX06} by $|K_{r}\times [-t_1, 0]|^2$
\[
\left(\frac{|A_{j+1}|}{|K_{r}\times [-t_1, 0]|}\right)^2 \leq \frac{4\gamma^2}{\nu_0\alpha^2}  \frac{|A_{j}\setminus A_{j+1}| }{|K_r\times [-2t_1, 0]|}.
\]
Let us take the summation for $j$ from $\ceil*{\log_2\frac{1}{\delta}}$ to $j^\ast$. Because $A_{j^\ast} \subset A_{j}$ for all $j < j^\ast$, we have
\[
(j^\ast -\ceil*{\log_2\frac{1}{\delta_1}})\left(\frac{|A_{j^\ast}|}{|K_{r}\times [-2t_1, 0]|}\right)^2 \leq \frac{4\gamma^2}{\nu_0 \alpha^2},\]
equivalently,
\begin{equation}\label{EPX07}
|A_{j^\ast}| \leq \sqrt{ \frac{4\gamma^2}{\nu_0\alpha^2 \, (j^\ast-\ceil*{\log_2\frac{1}{\delta_1}})} } \ \abs{K_{r}\times [-2t_1, 0]}.
\end{equation}
By choosing $j^\ast$ large enough such that
\[
\sqrt{ \frac{4\gamma^2}{\nu_0\alpha^2 \, (j^\ast-\ceil*{\log_2\frac{1}{\delta_1}})} } \leq \nu_1, \quad \text{equivalently } \quad j^\ast \geq \frac{16\gamma^2}{\nu_0^3 \nu_1^2}+\ceil*{\log_2\frac{1}{\delta_1}},
\]
we reach our conclusion \eqref{L:EPX:res} directly from \eqref{EPX07}.
\end{proof}

To conclude the second alternative, we take DeGiorgi type iteration with a upper-level set to say $v$ is strictly away from its maximum $\mu_{+}$ in the subcylinder $Q_{k,\rho}$ in $Q_{k, 4\rho}$ that both share the same vertex.

\begin{proposition}\label{L:SA} (Second alternative)
Let $v$ be a bounded non-negative weak solution of \eqref{main v} in $Q_{k, 4\rho}$. Suppose \eqref{case1}, \eqref{small mu-} and \eqref{L1:ass} hold.
Then there exists $\delta^\ast(\nu_0)<\frac{1}{4}$ such that either
\begin{equation}\label{L:FA:result3}
  k^{-2-\beta \frac{2(1+\kappa)}{q_1}} \rho^{d\kappa  }>{\delta^\ast}^2,
\end{equation}
or
\begin{equation}\label{L:SA:res}
  \esssup_{Q_{k, \rho} (\nu_0)} v(x,t) < \mu^{+} - \frac{\delta^\ast}{2}k.
\end{equation}
Here $Q_{k,\rho}(\nu_0)$ is given in \eqref{Q}.
\end{proposition}

\begin{proof}
Suppose \eqref{L:FA:result3} fails. Then following from Lemma \ref{L:EPX}, we know for any $\nu_1 \in ( 0,1 )$, there exists $\delta^\ast(\nu_1)$ such that
\begin{equation}\label{L:SA:1}
\left|\left\{Q_{k,2 \rho}({\nu_0}): v(x,t) > \mu_{+} - \delta^\ast k\right\}\right|  < \nu_1 |Q_{k, 2\rho}(\nu_0)|.
\end{equation}
We are going to select $\nu_1$ below which only depends on $\nu_0$.

The idea of the proof is similar to the first alternative.
Here we carry details of how to take the DeGiorgi iteration with an upper-level set.

Denote
\[h:=\delta^\ast k,\quad h_i:=\frac{h}{2}+\frac{h}{2^{i+1}}.\]
Let us consider a upper-level set $(v -\mu_{+} +h)_{+}$ and apply the local energy estimate in the form of \eqref{energy 1} from Proposition~\ref{P:EE}, with the sequences
\begin{gather*}
\rho_{i} = {\rho} + \frac{\rho}{2^{i}}, \quad
K_{i}= K_{\rho_i}, \quad Q_{i} =K_{i}\times (-\nu_0 k^{-\beta}
\rho_{i}^{2}, 0]
\end{gather*}
and a sequence of smooth cutoff functions
$\{ \zeta_{i}\}$ such that
\begin{equation*}
\zeta_{i} = \begin{cases}
            1 & \text{ in } Q_{i+1} \\
            0 & \text{ on the parabolic boundary of } Q_{i},
 \end{cases}
\end{equation*}
satisfying
\[
|\nabla \zeta_{i}| \leq \frac{2^{i+2}}{\rho}, \quad \text{ and } \quad
\partial_{t} \zeta_{i} \leq \frac{2^{2(i+2)}}{ \nu_0 k^{-\beta} \rho^{2}}.
\]

Let $v_{+}^{i} = (v-\mu_{+}+h_i)_{+}$, then it is trivial that $v_{+}^{i}$ means the set where
\[ \mu_{+} - h \leq \mu_{+} - h_{i} < v \leq \mu_{+} . \]
Moreover it follows
\begin{equation}\label{EE11}\begin{split}
&\quad \mu_+^{-\beta}\sup_{ - \nu_0 k^{-\beta} \rho_{i}^{2} \leq t \leq 0}\int_{K_i \times \{t\}} (v_+^{i})^2\zeta_{i}^2\,dx
 +  \iint_{Q_{i}} \abs{\nabla (v_+^{i}\,\zeta_{i})}^2 \,dx\,dt \\
&\leq 
C(\mu_+-h)^{-\beta}h^2 \frac{2^{i+2} k^{\beta}}{\rho^2}|A_{+}^{i}|
 +Ch^2 \frac{2^{2(i+1)}}{\rho^2} |A_{+}^{i}| \\
&\quad + C\mu_{+}^{2/m} \|B\|^{2}_{L_{t}^{2\hat{q}_1} L_{x}^{2\hat{q}_2} (Q_\rho)} \left[ \int_{-k^{-\beta}\rho_{i}^{2}}^{0} \left[A_{+}^{i}(t)\right]^{\frac{q_1}{q_2}} \,dt\right]^{\frac{2(1+\kappa)}{q_1}},
\end{split}\end{equation}
by denoting
\begin{align*}
A_{+}^{i}(\tau) := \{ x\in K_i : (v(x,\tau) +\mu_{+} -h_i )_{+} > 0 \}, \\
A_{+}^{i} := \{ (x,t)\in Q_i : (v +\mu_{+} -h_i )_{+} > 0 \}.
\end{align*}
The multiplication of $\mu_{+}^{\beta}$ to \eqref{EE11} and the choice of $h,k$ provide
\begin{equation}\label{EE12}\begin{split}
&\quad \sup_{ - \nu_0 k^{-\beta} \rho_{i}^{2} \leq t \leq 0}\int_{K_i \times \{t\}} (v_+^{i})^2\zeta_{i}^2\,dx
 +  k^{\beta}\iint_{Q_{i}} \abs{\nabla (v_+^{i}\,\zeta_{i})}^2 \,dx\,dt \\
&\leq C \frac{2^{2i} h^2k^{\beta}}{\rho^2} |A_{+}^{i}|
+ C \|B\|^{2}_{L_{t}^{2\hat{q}_1} L_{x}^{2\hat{q}_2} (Q_\rho)} k^{1+\frac{1}{m}}\left[ \int_{-\nu_0 k^{-\beta}\rho_{i}^{2}}^{0} \left[A_{+}^{i}(t)\right]^{\frac{q_1}{q_2}} \,dt\right]^{\frac{2(1+\kappa)}{q_1}},
\end{split}\end{equation}
which is a parallel estimate to \eqref{EE01} for the lower level set.
By following the similar method above for $v_{-}^{i}$, we reach the conclusion that
\[
|A_{+}^\infty| =|\{(x,t)\in K_{\rho}\times (-\nu_0k^{-\beta}\rho^2,0]\}|  = 0
\]
if $|A_1|$ is small enough only depending on $\nu_0$. Seeing from \eqref{L:SA:1}, this is equivalent to that $\nu_1$ is small enough.
The above statement is equivalent to
\[
\esssup_{Q_{k,\rho}(\nu_0)} v(x,t) < \mu_{+} - \frac{h}{2}.
\]
\end{proof}

\subsection{Proof of H\"{o}lder continuity}\label{SS:Holder}

Now we are ready to prove local H\"{o}lder continuity of $v$ a non-negative bounded weak solution of \eqref{main v} from the two alternatives Proposition~\ref{L:FA} and Proposition~\ref{L:SA}. We prove the following two lemmas concerning linear and nonlinear cases, providing the bounds of solutions' oscillations in small parabolic cylinders.

\begin{lemma}\label{L:main linear}
Suppose $m=1$ and \eqref{B} holds with $\kappa>0$ in \eqref{hatq12}.
For given positive constants $r$ and $\omega$, suppose also that $v$ is a non-negative bounded weak solution of \eqref{main v} in
\[
Q_{ r} = K_{r} \times (-r^2, 0]
\]
with $\essosc_{Q_{ r}} v \leq \omega$. Then there are positive constants $\eta$, $\lambda$, both less than $1$ and $C_0\geq 1$ such that for all $n\geq 1$
\begin{equation}\label{L:main:res}
\essosc_{Q_{ \lambda^n r}} v \leq C_0\eta^n.
\end{equation}

\end{lemma}

\begin{proof}
Let $r_0=r$. Then \eqref{L:main:res} is satisfied when $n=0$ by $C_0\geq \omega$ and the assumption. We proceed inductively from the initial value $r_0$ and find \[ \lambda^n r_0\leq r_n\leq \frac{r_0}{2^n}\]
such that
\begin{equation}
\label{lnlem idct}
    \essosc_{Q_{ r_n}} v \leq  C_0\eta^n
\end{equation}
for some $\lambda<1,\eta<1,C_0\geq 1$ to be determined
(see \eqref{con parameter})
and eventually conclude.

Suppose \eqref{lnlem idct} is satisfied for some $n\geq 0$.
We introduce the following notations
\[\mu_{n,+}:=\sup_{Q_n}v,\quad \mu_{n,-}:=\inf_{Q_n}v,\quad k_n:=\mu_{n,+}-\mu_{n,-}
\quad \text{ where }
Q_n:=Q_{r_n}\]
The next $r_{n+1}$ is generated depending on the following cases.
\begin{itemize}
    \item[Case 1:]  Let $\delta_*$ be from Proposition \ref{L:SA}. If \[(\frac{1}{2}k_n)^{-2}(\frac{1}{4}r_n)^{d\kappa}>\delta_*^2,\] the situation is in some sense better since the oscillation is under control.
    We can have
    \begin{equation}\label{C0 bound}
        \begin{aligned}
            osc_{Q_{r_n/2}}v &\leq k_n\leq c(\delta_*) r_n^{d\kappa/2}\\
            &\leq c2^{-nd\kappa/2}r_0^{d\kappa/2} \quad \text{( since $r_n\leq \frac{r_0}{2^n}$ )}
            \\
            &\leq C_0\eta^{n+1},
        \end{aligned}
    \end{equation}
    where the last inequality is true if we take $\eta>(\frac{1}{2})^{d\kappa/2}$ and $C_0$ large enough.
    In order to apply the preceding scheme, let $r_{n+1}=\frac{1}{2}r_n$, and we repeat until it falls into the other cases.

\item[Case 2:]  Suppose $(\frac{1}{2}k_n)^{-2}(\frac{1}{4}r_n)^{d\kappa}\leq \delta_*^2$ and let $\nu_0$ be from Proposition \ref{L:FA}. Furthermore we assume
\begin{equation}
    \label{L:ass}
    \left|\left\{ (x,t) \in Q_{\frac{1}{2}k_n, r_n} : v(x,t) \geq \mu_{-} + \frac{k_n}{2} \right\}\right| \leq \nu_0 \left|Q_{\frac{1}{2}k_n, r_n}\right|
\end{equation}
which is \eqref{L:FA:ass} with $k=\frac{1}{2}k_n, 4\rho=r_n$.
Notice that when $m=1$, $Q_{\omega,r}=Q_r$.  Thus by Proposition \ref{L:FA}
\[\inf_{Q_{r_n/2}}v>\mu_{n,-}+\frac{k_n}{4}.\]
Take $r_{n+1}=r_n/2$. We have
\begin{equation*}
    \begin{aligned}
    k_{n+1}&=osc_{Q_{n+1}}v= \mu_{n+1,+}-\mu_{n+1,-}\\
    &\leq \mu_{n,+}-\mu_{n,-}+\frac{k_n}{4}\\
    &\leq \frac{3}{4}osc_{Q_{n}}v\leq \eta k_n.
    \end{aligned}
\end{equation*}

\item[Case 3:] We are left with the case that $(\frac{1}{2}k_n)^{-2}(\frac{1}{4}r_n)^{d\kappa}\leq \delta_*^2$ and \eqref{L:ass} fails. Since
\[\mu_{n,-}+\frac{k_n}{2}=\mu_{n,+}-\frac{k_n}{2},\]
so the failure of \eqref{L:ass} corresponds to
\[\left|\left\{ Q_{r_n}: v(x,t) \geq \mu_{+} - \frac{k_n}{2} \right\}\right| \leq (1-\nu_0) \left|Q_{r_n}\right|.\]
Now the conditions in Proposition \ref{L:SA} are satisfied and we obtain
\[  \esssup_{Q_{\frac{1}{4}r_n} (\nu_0)} v(x,t) < \mu^{n,+} - \frac{\delta^\ast}{4}k_n\]
where $Q_{\frac{1}{4}r_n} (\nu_0) = K_{\frac{1}{4}r_n} \times (-\nu_0 (\frac{1}{4}r_n)^2, 0]
$.
And therefore
\begin{equation*}
    \begin{aligned}
    osc_{Q_{\frac{1}{4}r_n}(\nu_0)}v&\leq \mu_{n,+}-\frac{\delta^* k_n}{4}-\mu_{n,-}\\
    &\leq (1-\frac{\delta^*}{4})osc_{Q_{n}}v.
    \end{aligned}
\end{equation*}
Now we pick
\[r_{n+1}=\sqrt{\nu_0}\frac{r_n}{4},\quad k_{n+1}= (1-\frac{\delta^*}{4})k_n. \]
Since
\[Q_{n+1}=K_{\frac{\sqrt{\nu_0}}{4}r_n} \times \left(- (\frac{\sqrt{\nu_0}}{4}r_n)^2, 0\right]\subset K_{\frac{1}{4}r_n} \times (-\nu_0 (\frac{1}{4}r_n)^2, 0]=Q_{\frac{1}{4}r_n}(\nu_0),\]
we conclude
\[k_{n+1}=osc_{Q_{n+1}}v\leq \eta k_n\quad \text{ with  $\eta \geq 1-\frac{\delta^*}{4}$.}\]

\end{itemize}
In all, the conditions imposed on $\eta,\lambda$ are true with the choice in the following way:
\begin{equation}\label{con parameter}
\begin{aligned}
&    \eta=\max\{\frac{1+2^{-d\kappa/2}}{2},\,\frac{3}{4},\,1-\frac{\delta_*}{4}\},\\
&   \lambda=\frac{\sqrt{\nu_0}}{4}.
\end{aligned}
\end{equation}
And $C_0$ is chosen to be large enough (only depending on $\eta,\delta_*,d,\kappa,\omega$) such that \eqref{C0 bound} holds.
Finally, we conclude the proof.

\end{proof}

\begin{remark}
\label{rmk linear}

The conclusion of Lemma \ref{L:main linear} can be generalized to the uniform parabolic equations of the following forms with no difficulty:
\[u_t=\nabla\cdot(f(u,x,t)\nabla u)+\nabla\cdot(B u)\]
with $C_1\leq f\leq C_2$ for some $C_2>C_1>0$ for all $u,x,t$. The corresponding energy estimates are needed and the proofs only require a few modifications.
\end{remark}

Now we give the main lemma for the degenerate equation.

\begin{lemma}\label{L:main} (Main lemma for $v$).
Let $m > 1$, $r$, $k$, $\omega$ be given positive constants and $\beta = 1-\frac{1}{m}$. Suppose $m>1$ and \eqref{B} holds with $\kappa>0$ in \eqref{hatq12}.
Let $v$ be a non-negative bounded weak solution of \eqref{main v} in
\[
Q_{k, r} = K_{r} \times (-k^{-\beta}r^2, 0]
\]
with $\essosc_{Q_{k, r}} v \leq \omega$ . Then there are positive constants $\eta$ and $\lambda$, both less than $1$ and $C_0\geq 1$ such that for all $n\geq 0$
\begin{equation}
\essosc_{Q_{ \lambda^n r}} v \leq C_0\eta^n .
\end{equation}
\end{lemma}

\begin{proof}

We want $k$ to be not too small compared to $\omega$. So
if $osc_{Q_{k,r}} v>4k$, we take $k_0\geq k$ and, since $Q_{k_0,r}\subset Q_{k,r}$ it is possible that
\[osc_{Q_{k_0,r}} v\leq osc_{Q_{k,r}} v\leq 4k_0.\]

Next let $r_0=r$. We are going to find
\[0<\lambda< \frac{1}{2}<\eta<1,\quad C_1\geq 1, \quad\epsilon\in (0,1)\]
and pairs $(k_n,r_n)$ and show inductively that
\begin{equation}
\begin{aligned}
\label{lnlem idct2}
    &\essosc_{Q_{k_n,  r_n}} v \leq 4k_n,\\
    &\frac{1}{2} k_{n-1}\leq k_n\leq \max\{\eta k_{n-1}, C_1r_n^\epsilon\},\\
    &\lambda r_{n-1}\leq r_n\leq \frac{1}{2} r_{n-1}
\end{aligned}
\end{equation}
and eventually conclude. Here $\eta,\lambda$ will be given in \eqref{cond etalambda} and $C_1,\epsilon$ will be given in \eqref{cond epC1}. Comparing to the proof of Lemma \eqref{L:main linear}, here we do not take the oscillation itself but the quantity $k_n$ which is equivalent to the oscillation.

To do the induction, suppose \eqref{lnlem idct2} is satisfied for the pair $(k_{n},r_n)$.
We introduce the following notations
\[\mu_{n,+}:=\sup_{Q_n}v,\quad \mu_{n,-}:=\inf_{Q_n}v
\quad \text{ where }
Q_n:=Q_{k_n,r_n}\]
The next pair $(k_{n+1},r_{n+1})$ is generated depending on the following cases.
\begin{itemize}

    \item[Case 1:]   We want interpret $k_n$ as the oscillation of $v$ in $Q_n$. If this is not true and suppose
    \begin{equation}
        \label{mainlm c1}
         osc_{Q_n}v\leq 2k_n,
    \end{equation}
    we set $k_{n+1}=\frac{1}{2}k_n$ and $r_{n+1}=\frac{1}{2} {r_n}$. It follows from direct computations that $k_{n+1}^{-\beta}r_{n+1}^2\leq k_n^{-\beta}r_n^2$ and so $Q_{n+1}\subset Q_n$. We have
    \begin{equation*}
    \essosc_{Q_{n+1}} v \leq \essosc_{Q_{n}} v \leq 2k_n \leq 4 k_{n+1}.\end{equation*}

    \item[Case 2:]  Now let us suppose for the rest of the cases that
    \[2k_n\leq osc_{Q_n}v\leq 4k_n.\]
    Recall the notations used in Proposition \ref{L:SA}. If
    \begin{equation}
        \label{prop 4.5 cond}
        k_n^{-\gamma}(\frac{1}{4}r_n)^{d\kappa}>\delta_*^2,\quad\text{ with }\gamma:=2+\beta\frac{2(1+\kappa)}{q_1},
    \end{equation} the oscillation is under control.
    We set $k_{n+1}=k_n$, $r_{n+1}=\frac{1}{2} r_n$,
    By direct computations, we have
    \begin{equation*}
        \begin{aligned}
 osc_{Q_{k_{n+1},r_{n+1}}}v&\leq 4k_{n+1}=4k_n\\
 &\leq c(\delta_*) r_{n+1}^{d\kappa/\gamma} \quad \text{ (by \eqref{prop 4.5 cond}) }.\\
        \end{aligned}
    \end{equation*}
   The inductive step is proved if we take \begin{equation}
       \label{cond epC1}
       \epsilon:=d\kappa/\gamma,\quad C_1:= c(\delta_*)=4^{-d\kappa/\gamma}\delta_*^{-2/\gamma}.
   \end{equation}

\item[Case 3:] In this case, suppose $\mu_{n,-}\geq k_n$. Then inside $Q_n$, $ v\in [k_n,\mu_{n,+}]\subset [k_n,5 k_n]$. We can rescale the equation in the $x, t$ directions so as to obtain a
uniformly parabolic equation, to which we may apply the previous Lemma \ref{L:main linear} and remark \ref{rmk linear} to see that $v$ is uniformly H\"{o}lder continuous in $\frac{1}{2}Q_n$. The proof is then finished.

\item[Case 4:]  We are left with the case when \eqref{case1}, \eqref{small mu-} and \eqref{L:FA:result3} hold. Then the conditions of Proposition \ref{L:FA} and Proposition \ref{L:SA} are satisfied. We can proceed with the two alternatives as done in the (Case 2 \& Case 3) in the proof of previous Lemma \ref{L:main linear}.

Let us only sketch the proof. If further, we assume
\begin{equation}
\label{L2:ass}
     \left|\left\{ (x,t) \in Q_{k_n, r_n} : v(x,t) \geq \mu_{-} + {k_n} \right\}\right| \leq \nu_0 \left|Q_{k_n, r_n}\right|
\end{equation}
with $\nu_0$ from Proposition \ref{L:FA},
it follows from the proposition that
\[\inf_{Q_{k_n, \frac{1}{2}r_n}}v>\mu_{n,-}+\frac{k_n}{2}.\]
We take $r_{n+1}=\frac{1}{4}r_n, k_{n+1}=\frac{7}{8}k_n$. Then
\[Q_{n+1}=Q_{\frac{3}{4}k_n,\frac{r_n}{4}}\subset Q_{k_n,r_n/2}.\]
We have
\begin{equation*}
    \begin{aligned}
osc_{Q_{n+1}}v
&\leq \mu_{n,+}-\mu_{n,-}-\frac{k_n}{2}\\
&\leq 4k_n-\frac{k_n}{2}=\frac{7}{2}k_n\\
&=4k_{n+1}.
    \end{aligned}
\end{equation*}

Otherwise if \eqref{L2:ass} fails:
\begin{align*}
    \label{L3:ass}
&\quad  \quad   \left|\left\{ (x,t) \in Q_{n} : v \geq \mu_{+} - {k_n} \right\}\right|\geq\\
     &\left|\left\{ (x,t) \in Q_{n} : v \geq \mu_{-} + {k_n} \right\}\right|\geq (1-\nu_0) \left|Q_{n}\right|
\end{align*}
 We apply Proposition \ref{L:SA} to have
\[  \esssup_{Q_{k_n,\frac{1}{4}r_n} (\nu_0)} v(x,t) < \mu^{n,+} - \frac{\delta^\ast}{2}k_n.\]
And therefore
\begin{equation*}
    \begin{aligned}
    osc_{Q_{k_n,\frac{1}{4}r_n}(\nu_0)}v&\leq \mu_{n,+}-\frac{\delta^* k_n}{2}-\mu_{n,-}\leq \frac{7}{2}k_n.
    \end{aligned}
\end{equation*}
For some $\lambda$, we pick
\[\eta= 1-\frac{\delta^*}{4},\quad r_{n+1}=\lambda r_n,\quad k_{n+1}= \eta k_n. \]
In order to have $Q_{n+1}\subset Q_{k_n,\frac{1}{4}r_n}(\nu_0)$, we need $\lambda\leq \frac{1}{4}$ and
\[k_{n+1}^{-\beta}r_{n+1}^2=(\eta k_n)^{-\beta}(\lambda r_n)^2\leq \nu_0k_n^{-\beta}(\frac{r_n}{4})^2\]
which is satisfied if we take
$\lambda^2\leq \nu_0\frac{\eta^\beta}{16}$. Finally we can conclude that
\[osc_{Q_{n+1}}v\leq \eta 4k_n=4k_{n+1}.\]

\end{itemize}

In all, we proved that there is a sequence $k_n,r_n$ such that for some $\lambda,\eta<1$,
\[   \lambda r_n\leq r_{n+1}\leq \frac{1}{2}r_n,\quad  k_{n+1}\leq \max\{\eta k_n, C_1 r_n^\epsilon\}\]
and thus
\begin{align*}
    \lambda^n r_0\leq r_n\leq 2^{-n}r_0,\quad
    osc_{Q_{k_n,r_n}}v\leq  4k_n\leq \max\{c \eta^n,C_1 2^{-\epsilon n}r_0^\epsilon\}
\end{align*}
where
\begin{equation}\label{cond etalambda}
    \eta=\max\{\frac{7}{8},\,1-\frac{\delta^*}{4}\},\quad\lambda=\min\{\frac{1}{4},\frac{\sqrt{\nu_0 \eta^\beta}}{4}\}.
\end{equation}
Notice $k_n<1$ for $n$ large and
\[Q_{\lambda^n r_0}\subset K_{r_n}\times [0,- r_n^2]\subset K_{r_n}\times (0,-k_n^{-\beta} r_n^2]= Q_{k_n,r_n} .\]
Therefore, there are some $\tilde{\eta}(\eta,C_1,\epsilon)\in (0,1), C_0(C_1,k_0)$ such that
\[\osc_{Q_{\lambda^n r_0}}v\leq C_0\tilde{\eta}^n.\]

\end{proof}

{
From the main lemmas, Lemma~\ref{L:main linear} and Lemma~\ref{L:main}, we recover the H\"{o}lder continuity of $v$ of \eqref{main v} and then $u$ a non-negative bounded weak solution of \eqref{PME}.
}
\begin{theorem}\label{T:Holder:v} (H\"{o}lder continuity).
Let $u$ be a non-negative bounded weak solution of \eqref{PME} with $m \geq 1$ in $Q_1$. Then $u$ is uniformly H\"{o}lder continuous in $Q_\frac{1}{2}$.
\end{theorem}

\begin{proof}
Let $v:=u^{1/m}$ which then solves \eqref{main v} in $Q_1$.
Let $M:=\sup_{Q_1} v$ and fix any $(x_0,t_0)\in Q_{1/2}$. Without loss of generality we can assume it is $(0,0)$. We take $r\in(0,\frac{1}{2})$, and $\omega=M$ such that
\[Q_{\omega,r}\subset Q_\frac{1}{2}.\]
Then
\[osc_{Q_{\omega,r}}v\leq \omega.\]
We apply Lemma \ref{L:main linear} for the case when $m=1$ and Lemma \ref{L:main} for the case when $m>1$, to find
\begin{equation}
\essosc_{Q_{\eta^n k, \lambda^n r}} v \leq C(\omega)\eta^n  \quad \text{ for all $n\geq 0$ }
\end{equation}
holds for some positive constants $\eta<1$, $\lambda<1$ and $C(\omega)\geq 1$. This gives us the desired H\"{o}lder continuity of $v$ and then $u$.

\end{proof}


\section{Critical Case}\label{S:critical}

{

First, we recall the critical regime of the drift term $B \in L_{t}^{2\hat{q}_1} L_{x}^{2\hat{q}_2} (Q_1)$ where
\begin{equation}\label{hatq12 cri}
\frac{2}{\hat{q}_1} + \frac{d}{\hat{q}_2} = 2\quad  { \text{ where} \  \hat{q}_1\in (1,+\infty] \ \text{and} \ \hat{q}_2\in [d,+\infty).}
\end{equation}
Moreover assuming divergence-free condition on $B$, that is, $\nabla \cdot B = 0$, we will establish the uniform contitnuity of $u$, a non-negative weak solution of \eqref{PME}.

In the critical case, we do not expect the H\"{o}lder continuity of a weak solution of \eqref{PME} and \eqref{main v} for $B$ in general, see the counterexample Theorem~\ref{thm Ac} (\cite[Theorem~1.3]{KZ}).
However Theorem~\ref{thm Ac} leaves an open question to the possibility of weaker modulus of continuity of solutions.


In this section, we show that the divergence-free condition on $B$ is sharp to obtain the uniform continuity of a non-negative bounded weak solution of \eqref{PME} {in the critical region \eqref{hatq12 cri}}. There are two main results: first, Theorem~\ref{main cri} to prove the uniform continuity of $u$ assuming the divergence-free vector $B$ is in the critical regime and, second, Theorem~\ref{thm:lackofcont} to construct a counterexample failing the uniform continuity without the divergence-free condition on $B$.

}



\medskip

The continuity property we are going to show for the solutions will depend on $\varrho_B$, given in \eqref{def varrho}. Actually we have

\begin{theorem}\label{main cri}
{ Suppose $B$ is divergence-free in the critical regime, that is, $B \in L_t^{2\hat{q}_1}L_x^{2\hat{q}_2}(Q_1)$ with $\hat{q}_1,\hat{q}_2$ satisfying \eqref{hatq12 cri}.}  Let $u$ be a non-negative bounded weak solution of \eqref{PME} with $m \geq 1$ in $Q_1$. Then $u$ is uniformly continuous in $Q_\frac{1}{2}$ depending on $m,d, \hat{q}_1, \hat{q}_2$ and $\varrho_B(r)$.

{Moreover, we have }
\begin{equation}
    \label{cri conclu}
    \sup_{(x,t)\in Q_{\frac{1}{2}}}\essosc_{(x,t)+Q(\lambda^n)}u\leq C\, {\max\{\varrho^{\frac{m}{m-1}}_B(\gamma^n),\, \gamma^n\}}
\end{equation}
for some $\lambda<\gamma<1$, $C\geq 1$ and all $n\geq 1$.
\end{theorem}

The main difference of the proof (compared to the subcritical case: Theorem \ref{T:Holder:v}) lies in the first alternative. For example when $\kappa=0$, {the inequality } \eqref{EE07} no longer holds. More importantly, if $\kappa=0$, we can not deduce $Y_\infty=0$ based on the iteration of \eqref{EE08} and \eqref{EE9}.
To overcome the problem, we make use of the divergence-free drift based on a differently derived energy estimate, which is Proposition \ref{energy cri}.

\subsection{The alternatives}\label{SS:cri:alter}



For $v$ a non-negative bounded weak solution of \eqref{main v} and a fixed constant $\rho \in (0, \frac{1}{2})$, let us denote
\[\mu_+=\sup_{Q_{k,\rho}}v,\quad \mu_-=\inf_{Q_{k,\rho}}v\]
where $Q_{k,\rho} = K_{\rho} \times (- k^{-\beta} \rho^2, 0]$.


We give the first alternative {in the following proposition and the second alternative in Proposition~\ref{L:SA cri}.}

\begin{proposition}\label{L:FAcri}
Under the condition of Theorem \eqref{main cri}. Let $v$ be a bounded non-negative weak solution of \eqref{main v} in $Q_{k, \rho}$. Suppose
\begin{equation}\label{case1cri}
  4k\geq \mu_{+} - \mu_{-} \geq 2k.
\end{equation} and
\begin{equation}
    \label{small mu- cri}
    \text{ either }\quad  (\mu_-\leq k) \quad \text{ or }\quad  (\beta=0)\quad \text{ holds }.
\end{equation}
Then there exists a number $\nu_0 ,\delta \in (0,1) $ such that, if
\begin{equation}\label{L:FAcri:ass}
  \left|\left\{ (x,t) \in Q_{k, \rho} : v(x,t) \geq \mu_{-} + k \right\}\right| \leq \nu_0 \left|Q_{k, \rho}\right|,
\end{equation}
then either
\begin{equation}\label{cond cri}
 k^{-\beta}\|B\|_{{{L^{2\hat{q}_1}_t L^{ 2\hat{q}_2}_x}(Q_{k,\rho})}}> \delta.
\end{equation}
or
\begin{equation*}
  \essinf_{Q_{k, \rho/2} } v(x,t) > \mu_{-} + \frac{k}{2}.
\end{equation*}
\end{proposition}

\begin{proof}

Set $v_k^-:=(v - \mu_{-} - {k})_{+}$. We construct sequences $\{\rho_i\}$, $\{k_i\}$, $\{K_{i}\}$, and $\{Q_i\}$ such that
\begin{gather}
\rho_{i} = {\rho } + \frac{\rho }{2^{i}}, \quad \rho_0 = \rho , \ \rho_{\infty}= \frac{\rho}{2},\\
k_{i} =\frac{k}{2} + \frac{k}{2^{i+1}}, \quad k_{0}=k, \ k_{\infty} =\frac{k}{2},\\
K_{i}= K_{\rho_i}, \quad Q_{i} =K_{i}\times (- k^{-\beta}
\rho_{i}^{2}, 0].
\end{gather}
Moreover, we take a sequence of piecewise linear cutoff functions
$\{ \zeta_{i}\}_{i=0}^{\infty}$ such that
\[
\zeta_{i} = \begin{cases}
            1 & \text{ in } Q_{i+1} \\
            0 & \text{ on the parabolic boundary of } Q_{i},
 \end{cases}
\]
satisfying
\begin{equation}
    \label{zeta i}
|\nabla \zeta_{i}| \leq \frac{2^{i+2}}{\rho}, \quad
\partial_{t} \zeta_{i} \leq \frac{2^{i+2} k^{\beta}}{ \rho^{2}}.
\end{equation}

By Proposition \ref{energy cri}, since $k_i\sim \mu_+-k_i$ when $m\ne 1$ by the assumption,
\begin{equation*}\begin{split}
&\quad k_i^{-\beta}\sup_{-k_i^{-\beta}\rho_i^2 \leq t \leq 0}\int_{K_{\rho_i} \times \{t\}} (v_{k_i}^-)^2\zeta_i^2\,dxdt
 +  \iint |\nabla (v_{k_i}^-\,\zeta_i)|^2 \,dxdt \\
&\leq Ck_i^{1+1/m}\iint_{v_{k_i}^->0}|\zeta_i\,\partial_t\zeta_i|\,dxdt+Ck_i^2\iint_{v_{k_i}^->0} |\nabla\zeta_i|^2\,dxdt\\
&\quad\quad +Ck_i^{-2\beta} \|B\|^2_{{L^{2\hat{q}_1}_t L^{ 2\hat{q}_2}_x}(Q_i)}
 \left\|v_{k_i}^-\zeta_i\right\|^2_{L^{q_1}_tL^{q_2}_x(Q_i)}
\end{split}\end{equation*}
By \eqref{zeta i} and definitions of $\rho_i,k_i,\beta=\frac{m-1}{m}$,
\begin{equation}\label{14energy}\begin{split}
&\quad \sup_{-k_i^{-\beta}\rho_i^2 \leq t \leq 0}\int_{K_{\rho_i} \times \{t\}} (v_{k_i}^-)^2\zeta_i^2\,dxdt
 +  k_i^\beta\iint |\nabla (v_{k_i}^-\,\zeta_i)|^2 \,dxdt \\
&\leq Ck_i^2\frac{2^{i+2 }k^\beta}{\rho^2}|\{v_{k_i}^->0\}\cap Q_i|+Ck_i^{2+\beta}\frac{2^{2i+4}}{\rho^2}|\{v_{k_i}^->0\}\cap Q_i|\\
&\quad\quad +Ck_i^{-\beta} \|B\|^2_{{L^{2\hat{q}_1}_t L^{ 2\hat{q}_2}_x}(Q_i)}
 \left\|v_{k_i}^-\zeta_i\right\|^2_{L^{q_1}_tL^{q_2}_x(Q_i)}.
\end{split}\end{equation}

Now we take the change of variable that $ \bar{t} = k_i^{\beta} t \in
[- \rho_i^2, 0] .$ Also denote $\bar{v} = v(\cdot, \bar{t})$,
$\bar{\zeta}_i = \zeta_i(\cdot, \bar{t})$, $\bar{v}_{i}=v^-_{k_i}(\cdot,\bar{t}) $ and $\overline{Q}_{i} = K_{i}
\times (- \rho_{i}^{2}, 0]$. Also for simplicity, denote
\[
\bar{A}_{i}:= \bar{A}^{-}_{k_i, \rho_i} = |\bar{Q}_{i} \cap \{ \bar{v}_{i} > 0\}|.
\]

Then \eqref{14energy} gives
\begin{equation}
\label{bd V1}
\begin{split}
&\quad \|\bar{v}_i\,\bar{\zeta}_i\|^2_{V^2(\bar{Q}_i)}:= \sup_{ -  \rho_{i}^{2}\leq \bar{t} \leq 0} \, \int_{K_{\rho_i}\times\{t\}}  (\bar{v}_i\,\bar{\zeta}_i)^2\,dx
 +  \iint |\nabla \bar{v}_i \bar{\zeta}_i)|^2 \,dxd\bar{t} \\
&\leq Ck_i^2\frac{2^{i }}{\rho^2}\bar{A}_{i}+Ck_i^{2}\frac{4^{i}}{\rho^2}\bar{A}_{i} +Ck_i^{-2\beta} \|B\|^2_{{L^{2\hat{q}_1}_t L^{ 2\hat{q}_2}_x}(Q_i)}
 \left\|\bar{v}_i\,\bar{\zeta}_i\right\|^2_{L^{q_1}_tL^{q_2}_x(Q_i)}.
\end{split}\end{equation}

Now we apply Sobolev
embedding theorem (Theorem~\ref{T:Sobolev}) for $\bar{v}_i\bar{\zeta}_i$ from which we calculate
\begin{equation*}
 \|\bar{v}_i\,\bar{\zeta}_i\|^2_{L^2(\bar{i})}\leq |\bar{A}_{i}|^{\frac{2}{d+2}}\|\bar{v}_i\,\bar{\zeta}_i\|^2_{V^2(\bar{Q}_i)}.
\end{equation*}

In the set $\{ \bar{v}_{i+1}> 0 \}$, we observe
that
$$ \bar{v}_i=(\bar{v} - \mu_{-} - k_{i})_{-} \geq k_{i} - k_{i+1} = \frac{k}{2^{i+2}}. $$
Then the combination of the above leads to the following inequality:
\begin{equation*}\begin{split}
&\frac{k^{2}}{4^{(i+2)}} \left|\bar{A}_{i+1}\right| \leq C\left(\frac{2^{i} k^{2}}{ \rho^2} + \frac{ 4^{i}k^{2}}{\rho^2}\right) \left|\bar{A}_{i}\right|^{1+\frac{2}{d+2}} \\
  &\quad + C k_i^{-2\beta} \|B\|^2_{{L^{2\hat{q}_1}_t L^{ 2\hat{q}_2}_x}(Q_i)}
 \left\|\bar{v}_i\,\bar{\zeta}_i\right\|^2_{L^{q_1}_tL^{q_2}_x(Q_i)}|\bar{A}_i|^{\frac{2}{d+2}}
  .
\end{split}\end{equation*}

We note that
$$ \rho^2 \sim |\overline{Q}_{i}|^{\frac{2}{d+2}}, \quad |K_{i}| \sim |\overline{Q}_{i}|^{\frac{d}{d+2}}, \quad |\bar{Q}_{i}| = c(d) |\bar{Q}_{i+1}|. $$
Now multiply $2^{2(i+2)}/ k^{2}$ and divide by $|\bar{Q}_{i+1}|$ that yields
\begin{equation}\label{ineq frac}
\frac{|\bar{A}_{i+1}|}{|\bar{Q}_{i+1}|}
\leq C_0  16^{i}\left(\frac{\left|\bar{A}_{i}\right|}{|\bar{Q}_i|}\right)^{1+\frac{2}{d+2}} + C_14^ik^{-2\beta-2} \|B\|^2_{{L^{2\hat{q}_1}_t L^{ 2\hat{q}_2}_x}(Q_i)}
 Z_i\left(\frac{\left|\bar{A}_{i}\right|}{|\bar{Q}_i|}\right)^{\frac{2}{d+2}}
\end{equation}
where $C_0$ and $C_1$ are positive constants depending on $d, m$ and
\[
{Z}_i = \frac{1}{|K_i|}\left\|\bar{v}_i\,\bar{\zeta}_i\right\|^2_{L^{q_1}_tL^{q_2}_x(Q_i)}.
\]

Now by Proposition~\ref{P:admissible} (Soblev embedding) {and \eqref{bd V1}}, we are able to obtain
\begin{align*}
    Z_{i} &\leq \frac{1}{|K_i|}  \|\bar{v}_i\,\bar{\zeta}_i\|^2_{V^2(\bar{Q}_i)}\\
&\leq \frac{C}{|K_i|}k_i^{2}\frac{4^{i}}{\rho^2}\bar{A}_{i} +\frac{C}{|K_i|}k_i^{-2\beta} \|B\|^2_{{L^{2\hat{q}_1}_t L^{ 2\hat{q}_2}_x}(Q_i)}
 \left\|\bar{v}_i\,\bar{\zeta}_i\right\|^2_{L^{q_1}_tL^{q_2}_x(Q_i)}\\
 &\leq C4^i k^2 \left(\frac{\left|\bar{A}_{i}\right|}{|\bar{Q}_i|}\right)+Ck^{-2\beta}\|B\|^2_{{L^{2\hat{q}_1}_t L^{ 2\hat{q}_2}_x}(Q_i)}Z_i.
\end{align*}

Recall \eqref{cond cri}, if we take $\delta$ to be small enough only depending on universal constants, we can assume
\begin{equation}
Ck^{-2\beta}\|B\|^2_{{L^{2\hat{q}_1}_t L^{ 2\hat{q}_2}_x}(Q_i)} \leq 1/2
\end{equation}
which leads to
\[Z_i\leq C4^i k^2 \left(\frac{\left|\bar{A}_{i}\right|}{|\bar{Q}_i|}\right).\]
By \eqref{ineq frac}
\begin{align*}
    \frac{|\bar{A}_{i+1}|}{|\bar{Q}_{i+1}|}
&\leq C_0  16^{i}\left(\frac{\left|\bar{A}_{i}\right|}{|\bar{Q}_i|}\right)^{1+\frac{2}{d+2}} + CC_1 16^ik^{-2\beta} \|B\|^2_{{L^{2\hat{q}_1}_t L^{ 2\hat{q}_2}_x}(Q_i)}
 \left(\frac{\left|\bar{A}_{i}\right|}{|\bar{Q}_i|}\right)^{1+\frac{2}{d+2}}\\
 &\leq C_2  16^{i}\left(\frac{\left|\bar{A}_{i}\right|}{|\bar{Q}_i|}\right)^{1+\frac{2}{d+2}}
\end{align*}

Therefore by Lemma~\ref{L:iter2} with $Y_n=\frac{|\bar{A}_{n}|}{|\bar{Q}_{n}|}$ and $Z_n\equiv 0$, if $\frac{|\bar{A}_{0}|}{|\bar{Q}_{0}|}$ is small enough only depending on $d,m$ and then $\frac{|\bar{A}_{i+1}|}{|\bar{Q}_{i+1}|}$ tends to zero as $i\to \infty$.
And the smallness of  $\frac{|\bar{A}_{0}|}{|\bar{Q}_{0}|}$ is guaranteed if $\nu_0$ is small.
\end{proof}


{
Now we deliver the second alternative when the assumption for the first alternative \eqref{L:FAcri:ass} fails. The proof is a combination of the proof for Propositon~\ref{L:SA}, the second alternative for the subcritical case and the previous Proposition.
}

\begin{proposition}
\label{L:SA cri}
Suppose $\nabla\cdot B=0$ and $B\in L_t^{2\hat{q}_1}L_x^{2\hat{q}_2}(Q_1)$ with
$\frac{2}{\hat{q}_1} + \frac{d}{\hat{q}_2} = 2.$
Let $v$ be a bounded non-negative weak solution of \eqref{main v} in $Q_{k, \rho}$. Suppose
\eqref{case1cri}, \eqref{small mu- cri} and
\[ \left|\left\{ (x,t) \in Q_{k, \rho} : v(x,t) \geq \mu_{-} + k \right\}\right| \geq \nu_0 \left|Q_{k, \rho}\right|.
\]
Then there exists $\delta^\ast(\nu_0)<\frac{1}{4}$ such that either
\begin{equation*}
{ k^{-\beta}}\|B\|_{{{L^{2\hat{q}_1}_t L^{ 2\hat{q}_2}_x}(Q_{k,\rho})}}\leq \delta^\ast.
\end{equation*}
or
\begin{equation*}
  \esssup_{Q_{k, \rho} (\nu_0)} v(x,t) < \mu^{+} - \frac{\delta^\ast}{2}k.
\end{equation*}
\end{proposition}

\subsection{Sketch of the proof for Theorem \ref{main cri}}\label{SS:cri:proof}

The proof is similar to the one of Lemmas \ref{L:main linear}, \ref{L:main} and Theorem \ref{T:Holder:v}.

We only consider the case $m>1$ here. As before, let $v=u^{1/m}$ and consider $(0,0)$ point.
As before we are going to find pairs $(k_n,r_n)$ and proceed the proof inductively. Let us start with some $k_0,r_0<\frac{1}{4}$ such that
\[osc_{Q_{k_0,r_0}}v\leq 4k_0.\]

We need to consider the modulus of continuity generated by $B$. Similarly as in \eqref{def varrho}, set
\[\varrho_n:=\left\|B\right\|_{L_t^{2\hat{q}_1}L_x^{2\hat{q}_2}({Q_{k_n,r_n}})}.\]
We are going to show inductively that for some $\eta\in (\frac{1}{2},1)$ and $\lambda<\frac{1}{2}$, there are
\begin{equation}
\label{lnlem idct2cri}
    \begin{aligned}
        & \essosc_{Q_{k_n,  r_n}} v \leq 4k_n,\\
        & \lambda r_{n-1}<r_n\leq \frac{1}{2} r_{n-1},\\
        &\frac{1}{2}k_{n-1}\leq k_n\leq \max\{\eta k_{n-1},  (\frac{\varrho_{n-1}}{\delta})^{1/\beta}\}
    \end{aligned}
\end{equation}
where $\delta$ comes from Proposition \ref{L:FAcri}.

To do the induction, suppose \eqref{lnlem idct2cri} is satisfied for the pair $(k_{n},r_n)$.
We introduce the following notations
\[\mu_{n,+}:=\sup_{Q_n}v,\quad \mu_{n,-}:=\inf_{Q_n}v
\quad \text{ where }
Q_n:=Q_{k_n,r_n}\]
The next pair $(k_{n+1},r_{n+1})$ is generated depending on the following cases.
\begin{itemize}

    \item[Case 1:]   If
    $   2k_n\geq osc_{Q_n}v$,
   let $k_{n+1}=\frac{k_n}{2}$ and $r_{n+1}=\frac{r_n}{2}$. We have
    \begin{equation*}
    \essosc_{Q_{n+1}} v \leq \essosc_{Q_{n}} v \leq 2k_n \leq 4 k_{n+1}.\end{equation*}

    \item[Case 2:]  Now let us suppose for the rest of the cases that
    \[2k_n\leq osc_{Q_n}v\leq 4k_n.\]
    If
   $
      \delta  k_n^{\beta}< \varrho_n=\|B\|_{{{L^{2\hat{q}_1}_t L^{ 2\hat{q}_2}_x}(Q_{n})}},$
     the oscillation is under control.
    We set
    \[k_{n+1}=k_n, r_{n+1}=\frac{1}{2}r_n.\]
Then it is direct that
    \begin{equation*}
        \begin{aligned}
 osc_{Q_{n+1}}v&\leq 4k_{n+1}\quad\text{and}\quad k_{n+1}\leq  (\frac{\varrho_n}{\delta})^{1/\beta}.
        \end{aligned}
    \end{equation*}

\item[Case 3:]
    Now if $\mu_{n,-}\geq k_n$, then inside $Q_n$, $ v\in [k_n,\mu_{n,+}]\subset [k_n,5 k_n]$. We may then rescale the equation in the $x, t$ directions so as to obtain a
uniformly parabolic equation, to which we may apply the previous result and remark \ref{rmk linear} to see that $v$ is uniformly H\"{o}lder continuous in $\frac{1}{2}Q_n$. The proof is then finished.

\item[Case 4:] From the previous Cases 1-3, we are left with the situation when \eqref{case1cri}, \eqref{small mu- cri} and \eqref{cond cri} hold in $Q_n$. Then the conditions of Proposition \ref{L:FAcri} and Proposition \ref{L:SA cri} are satisfied. We can proceed with the two alternatives as done in the (Case 2 \& Case 3) in the proof of Lemma \ref{L:main linear}. The alternatives yield: there exist $\lambda<\frac{1}{4},\eta\in(\frac{1}{2},1)$ such that
for $r_{n+1}:=\lambda r_n$ and $k_{n+1}:=\eta k_n$, we have
\[osc_{Q_{n+1}}v\leq 4k_{n+1}.\]

\end{itemize}

Finally we claim that we proved \eqref{lnlem idct2cri}. By definition, $Q_n=Q_{k_n,r_n}$ is a subset of either $Q_{r_n}$ or $Q_{k_n^{-{\beta}/{2}}r_n}$.
It follows from the inductive assumptions on $k_n,r_n$, we get
\[r_n\leq 2^{-n}r_0\to 0\quad \text{ and }\quad \frac{r_n^2}{k_n^\beta}\leq ({2^{\beta-2}})^n\frac{r_0^2}{k^\beta_0}\to 0 \quad\text{ as }n\to\infty.\]
Write $\gamma_n^2:=\max\{({2^{\beta-2}})^n\frac{r_0^2}{k^\beta_0}, 4^{-n}r_0^2\}$ and therefore for some $c>0$
\begin{equation}
    \label{Qn est}
Q_{n}\subset Q_{\gamma_n}=K_{\gamma_n}\times (-\gamma_n^2,0]\quad\text{ and }\quad \gamma_n\leq c2^{(\frac{\beta}{2}-1)n}.
\end{equation} Using the notation \eqref{def varrho}, we have
\[ \varrho_n\leq \varrho_B(\gamma_n)\to 0 \quad\text{ as }n\to\infty.\]

Hence \eqref{lnlem idct2cri}, \eqref{Qn est} lead to
\[\essosc_{Q(\lambda^n)}v\leq C \,{\max\{\varrho^{1/\beta}_B(\gamma^n),\gamma^n\}}\]
for some $\gamma\in(0,1)$ and $C(\delta)\geq 1$. From this it is not hard to verify \eqref{cri conclu}.

\subsection{{Loss of the uniform continuity}}

In the previous sections (Section~\ref{SS:cri:alter} and Section~\ref{SS:cri:proof}), we complete the proof of Theorem~\ref{main cri} providing a non-negative bounded weak solution of \eqref{PME} is uniformly continuous as long as the divergence-free vector field $B$ is in the critical case. Indeed, we proved that solutions yield weaker norm of continuity depending on the modulus continuity of the drift $\varrho_B(r)$, defined in \eqref{def varrho}.

To claim the sharpness of the divergence-free condition on $B$ giving the uniform continuity of a weak solution of \eqref{PME}, we construct a counterexample distinguishing the regularity of a weak solution of \eqref{PME} corresponding to the divergence-free and general drift terms, respectively.

First, let us recall Theorem~5.2 \cite{KZ} providing the failure of the boundedness in the critical case.
However, the theorem does not take the decay of the norm of $B$ in small scales into consideration.
In order to have a more complete picture, as one extension of Theorem~5.2 \cite{KZ}, we will show that the loss of continuity can be indeed independent of the modulus of continuity $\varrho_B(r)$.




Consider spatial drift $B=B(x)$ in the critical region \eqref{hatq12 cri}. For $2\hat{q}_1=\infty$ and $2\hat{q}_2=d$ in \eqref{def varrho}, we write
\[\varrho_B(r)=\sup_{x_0}\left\|B\right\|_{L^{d}((x_0+K_{r})\cap K_1)}.\]

\begin{theorem}\label{thm:lackofcont}
 There exist sequences of vector fields $\{B_n\}_n$, which are uniformly bounded in $L^d(\R^d)$ and
 \[\varrho_{B_n}(r)\leq \omega(r) \quad \text{for some modulus of continuity }\omega,\]
 along with sequences of uniformly bounded functions $\{u_{n}\}$ in $K_1$ which are stationary solutions to \eqref{PME} with $B=B_n$ such that, they do not share any common mode of continuity.
\end{theorem}

Here to illustrate the main idea, we choose not to state our theorem in the form the dynamic solutions, but just simple stationary solutions.
\begin{proof}

For each $N\geq 4$, let $\phi_N:[0,\infty)\to \mathbb{R}$ be such that
\begin{itemize}
\item[1.] \begin{equation*}
    \phi_N(r)=\left\{\begin{aligned}
    &\ln\ln \frac{1}{r}& \text{ when }&  r\in (\frac{1}{N},\frac{1}{e})\\
& \phi_N(r)=2\ln\ln N &  \text{ when }&  r \in [ 0,\frac{1}{10N});
    \end{aligned}\right.
\end{equation*}

\item[2.] $ -\frac{N}{\ln N}\leq \frac{d}{dr}\phi_N\leq 0$ when $r\leq \frac{1}{e}$ and $| \frac{d}{dr}\phi_N|\leq 2e$ when $r\geq \frac{1}{e}$.
\end{itemize}
It is not hard to see such $\phi_N$ exists for each $N\geq 4$.

Define
\[\Phi_N(x):=-\frac{\phi_N(|x|)}{\ln\ln N}\]
and then it is trivial to check that $\nabla\Phi_N$ is locally uniformly bounded in $L^d(\R^d)$.

In addition, we show that the modulus of the local $L^d$ norm of the vector field $\nabla\Phi_N$ is uniformly bounded by one modulus of continuity $\omega$ for all $N$.
Since $\Phi_N$ is time-independent and mostly concentrated near the origin, we have
\begin{align*}
    \varrho_N(r):=\varrho_{(\nabla \Phi_N)}(r)=\sup_{x_0\in \R}\|\nabla \Phi_N\|_{L^d(K_r(x_0))}=\|\nabla \Phi_N\|_{L^d(K_r(0))}.
\end{align*}
Note that
\[|\nabla \Phi_N|=\frac{m}{m-1}|\phi_N'|/\ln\ln N.\]
Therefore by definition, $\varrho_N(r)=0$ when $r\leq \frac{1}{10N}$.
When $\frac{1}{10N}\leq r\leq \frac{2}{N}$, we have
\begin{equation*}
    \begin{aligned}
 \varrho_N(r)^d&\lesssim \int_{0}^r |\phi'_N(s)/\ln\ln N|^{d}s^{d-1}ds\\
 &\lesssim (\ln\ln N )^{-d}\int_0^r (\frac{N}{\ln N})^{d}s^{d-1}ds\quad \text{(by property 2 of $\phi_N$)}\\
 &\approx (\ln\ln N \ln N )^{-d}(Nr)^d\\
 &\leq C(d,m)(\ln\ln \frac{1}{r} \ln \frac{1}{r})^{-d}\quad \text{( since $r\geq\frac{1}{10N}$)}.
    \end{aligned}
\end{equation*}
When $\frac{2}{N}\leq r\leq \frac{1}{e}$, by property 1,
\begin{align*}
    \varrho^d_N(r)&\lesssim (\ln\ln N \ln N)^{-d}+\int_{\frac{2}{N}}^{r} |\frac{1}{s\ln s}|^ds^{d-1}ds \\
    &\lesssim (\ln\ln N \ln N)^{-d}+(d-1) |\ln s|^{1-d}\big|_{\frac{2}{N}}^r \\
    &\leq C({d,m})  (\ln\ln \frac{1}{r} \ln \frac{1}{r})^{-d}+C({d,m})(\ln \frac{1}{r})^{-d+1}.
\end{align*}
Lastly when $r\geq \frac{1}{e}$, since $\phi_N'$ is bounded,
\[\varrho^d_N(r)\leq \varrho^d_N(\frac{1}{e})+C r^d.\]

From the above computations, if defining
\[\omega^d(r):=C (\ln\ln \frac{1}{r} \ln \frac{1}{r})^{-d}+C(\ln \frac{1}{r})^{-d+1}+Cr^d\]
for some $C$ large enough only depending on $d,m$, we conclude that $\varrho_{N}(r)$ is indeed bounded by $\omega(r)$, a modulus of continuity independent of $N$.

\medskip

Now we are left to show that the corresponding solutions (to the equation with drift $\Phi_N$) do not share any modulus of continuity.
We start with the case when $m>1$. Notice that
\[u_N:=\frac{m-1}{m}\left(\frac{\phi_N(|x|)-\ln\ln\ln\ln N}{\ln\ln N}\right)_+^\frac{1}{m-1}\]
solves the stationary equation
\[\nabla\cdot\left( u_N \left(\nabla \left(\frac{m}{m-1} u_N^{m-1}+ \Phi_N\right)\right)\right)=\Delta (u_N^m)+\nabla\cdot (u_N\nabla \Phi_N)=0.\]
Also by definition
\[u_N(0)=\left(\frac{2\ln\ln N-\ln\ln\ln\ln N}{\ln\ln N}\right)^\frac{1}{m-1}\sim 1,\]
while it is only supported inside $K_{R}(0)$ with $R=\frac{1}{\ln\ln N}$. These two show that the sequence of bounded stationary solutions $u_N$ can not share any modulus norm of continuity.

When $m=1$, note that
\[u_{N}:=exp\left(\frac{\phi_N(|x|)-\ln\ln\ln\ln N}{\ln\ln N}\right)\]
is one stationary solution to
\[\nabla \cdot\left(u_N \nabla \left(\ln u_N+\Phi_N\right)\right)=\Delta u_N+\nabla\cdot (u_N\nabla \Phi_N)=0.\]
But the following two hold
\[u_N(0)=e^2+O(\ln\ln\ln\ln N/\ln\ln N),\]
\[u_N(\frac{1}{\ln\ln N})=exp(0)=1,\]
which again give the loss of uniform continuity.

\end{proof}

\end{document}